\documentclass{siamart0516}
\usepackage{draftwatermark}

\usepackage{amsmath,amsfonts,amssymb}
\usepackage{draftwatermark}
\usepackage{stmaryrd}
\usepackage{boxedminipage}
\usepackage{algorithmic}
\usepackage{url}
\usepackage{verbatim}
\usepackage{subfigure}
\usepackage{tikz}
    \usetikzlibrary{math}
    \usetikzlibrary{matrix}

\usepackage{comment}
\usepackage{mathabx}

\usepackage[mathscr]{eucal}

\usepackage{amsbsy}

\usepackage{bm}

\usepackage{paralist}

\usepackage{xspace}

\newtheorem{example}[theorem]{Example}
\newtheorem{remark}{Remark}

\newcommand{\V}[1]{{\bm{\mathbf{\MakeLowercase{#1}}}}} 

\newcommand{\M}[1]{{\bm{\mathbf{\MakeUppercase{#1}}}}} 
\newcommand{\Mhat}[1]{{\bm{\hat \mathbf{\MakeUppercase{#1}}}}} 

\newcommand{\T}[1]{\boldsymbol{\mathscr{\MakeUppercase{#1}}}} 

\newcommand{\TA}{\T{A}}

\newcommand{\TC}{\T{C}}
\newcommand{\TX}{\T{X}}

\newcommand{\TG}{\T{G}}

\newcommand{\TT}{\T{T}}
\newcommand{\TP}{\T{P}}
\newcommand{\TTT}{ \T{T}}

\newcommand{\Vu}{\V{u}}

\newcommand{\MH}{\M{H}}
\newcommand{\MG}{\M{G}}
\newcommand{\MZ}{\M{Z}}

\newcommand{\MA}{\M{A}}
\newcommand{\MF}{\M{F}}
\newcommand{\MT}{\M{T}}
\newcommand{\MX}{\M{X}}
\newcommand{\MS}{\M{S}}

\newcommand{\MV}{\M{V}}
\newcommand{\MU}{\M{U}}
\newcommand{\MD}{\M{D}}
\newcommand{\MB}{\M{B}}
\newcommand{\MP}{\M{P}}
\newcommand{\MQ}{\M{Q}}
\newcommand{\ML}{\M{L}}
\newcommand{\MW}{\M{W}}
\newcommand{\MCC}{\M{C}}

\newcommand{\MI}{\M{I}}
\newcommand{\MY}{\M{Y}}

 \newcommand{\bea}{ \left[ \begin{matrix} }
 \newcommand{\eea}{ \end{matrix} \right] }
\newcommand{\diag}{ {\tt diag}}

\newcommand{\sq}{ {\tt sq}}
\newcommand{\twist}{ {\tt twist}}
\newcommand{\struct}{ {\tt struct}_{\mc{E}}}
\newcommand{\structLa}{ {\tt struct}_{\mc{E}^{(1)}}}
\newcommand{\structLb}{ {\tt struct}_{\mc{E}^{(2)}}}
\newcommand{\structLc}{ {\tt struct}_{\mc{E}^{(3)}}}

\newcommand{\trace}{{\tt trace}}

\newcommand{\cTT}{ \mathcal{T}}
\newcommand{\cM}{ \mathcal{M}}

\newcommand{\mb}[1]{\mathbb{#1}}
\newcommand{\mc}[1]{\mathcal{#1}}
\newcommand{\bmat}[1]{\begin{bmatrix} #1\end{bmatrix}}

\newcommand{\arvind}[1]{\textcolor{blue}{Arvind: #1}}

\SetWatermarkText{ }
\SetWatermarkScale{3}

\definecolor{blue}{rgb}{0,0,1}
\definecolor{red}{rgb}{1,0,0}
\definecolor{green}{rgb}{.5,.8,.5}

\newcommand\MEK[1]{\textcolor{red}{MEK: #1}}

\usepackage{pgfkeys} 	
\usepackage{ifthen} 		

\pgfkeys{
/tensor/.cd, 
dim1/.initial = 1,		dim1/.get = \dimOne,		dim1/.store in = \dimOne,
dim2/.initial = 1,		dim2/.get = \dimTwo,		dim2/.store in = \dimTwo,
dim3/.initial = 1,		dim3/.get = \dimThree,	dim3/.store in = \dimThree,
xshift/.initial = 0, 	xshift/.get = \xShift, 		xshift/.store in = \xShift,
yshift/.initial = 0, 	yshift/.get = \yShift, 		yshift/.store in = \yShift,
xspec/.initial = 0, 	xspec/.get = \xSpec, 		xspec/.store in = \xSpec,
yspec/.initial = 0, 	yspec/.get = \ySpec, 		yspec/.store in = \ySpec,
scale/.initial = 1, 	scale/.get = \myScale, 	scale/.store in = \myScale,
fill/.initial = white, 	fill/.get = \myFill, 		fill/.store in = \myFill,
back edges/.initial = 0, 	back edges/.get = \myBack, 	back edges/.store in = \myBack,
slice type/.initial = none, 		slice type/.get = \sliceType, 			slice type/.store in = \sliceType, 
number of slices/.initial = 1, 	number of slices/.get = \nSlices, 		number of slices/.store in = \nSlices,
slice width/.initial = 1, 		slice width/.get = \sWidth, 				slice width/.store in = \sWidth, 
}

\newcommand{\tensor}[1][]{\@tensor[#1]}
\def\@tensor[#1] (#2,#3) #4; {{ 

\pgfkeys{/tensor/.cd,#1}

\def\depthScale{0.5} 

\pgfmathsetmacro{\numSlicesMinusOne}{\nSlices-1}
\pgfmathsetmacro{\numSlicesPlusOne}{\nSlices+1}

\pgfmathsetmacro{\sliceLength}{\myScale*\dimOne}

\ifthenelse{\equal{\sliceType}{lateral}}
	{
	
	\pgfmathsetmacro{\sliceWidth}{\myScale*\sWidth*0.9*\dimTwo/\nSlices}
	\pgfmathsetmacro{\sliceGap}{\myScale*\dimTwo/(\nSlices-1) - \nSlices*\sliceWidth/(\nSlices-1)}
	\pgfmathsetmacro{\sliceDepth}{\myScale*\dimThree}
	
	} 
	{
	\ifthenelse{\equal{\sliceType}{frontal}}
		{
		
		\pgfmathsetmacro{\sliceDepth}{\myScale*\sWidth*0.9*\dimThree/\nSlices}
		\pgfmathsetmacro{\sliceGap}{\myScale*\dimThree/(\nSlices-1) - \nSlices*\sliceDepth/(\nSlices-1)}
		\pgfmathsetmacro{\sliceWidth}{\myScale*\dimTwo}
	
		}
		{
		\pgfmathsetmacro{\sliceWidth}{\myScale*\dimTwo}
		\pgfmathsetmacro{\sliceDepth}{\myScale*\dimThree}
		}

	}

\def\xFront{#2 + \xShift}	
\def\yFront{#3 + \yShift}
\def\xBack{#2 + \xShift + \depthScale*\sliceDepth + \xSpec*\sliceDepth}
\def\yBack{#3 + \yShift + \depthScale*\sliceDepth + \ySpec*\sliceDepth}

\def\aFront{(\xFront, \yFront)}
\def\bFront{(\xFront, \yFront + \sliceLength)}
\def\cFront{(\xFront + \sliceWidth, \yFront + \sliceLength)}
\def\dFront{(\xFront + \sliceWidth, \yFront)}

\def\aBack{(\xBack, \yBack)}
\def\bBack{(\xBack, \yBack + \sliceLength)}
\def\cBack{(\xBack + \sliceWidth, \yBack + \sliceLength)}
\def\dBack{(\xBack+ \sliceWidth, \yBack)}

\ifthenelse{\NOT\equal{\myFill}{nofill}}
	{
	\def\tempTensor{
		\fill[\myFill!25] \bFront -- \bBack -- \cBack -- \cFront -- cycle; 
		\fill[\myFill!75] \dFront -- \dBack -- \cBack -- \cFront -- cycle; 
		\fill[\myFill!50] \aFront rectangle \cFront;  
	
		\draw \aFront rectangle \cFront; 
		\draw \bFront -- \bBack; 
		\draw \cFront -- \cBack;
		\draw \dFront -- \dBack;
	
		\draw \bBack -- \cBack;
		\draw \cBack -- \dBack;
		}
	}
	{ 

	\def\tempTensor{
		\draw \aFront rectangle \cFront; 
		
		\ifthenelse{\NOT\equal{\myBack}{0}}
		{
			\draw[dashed] \bBack -- \aBack -- \dBack;
		}{}
		
		\draw \dBack -- \cBack -- \bBack;

		\ifthenelse{\NOT\equal{\myBack}{0}}
		{
			\draw[dashed] \aFront -- \aBack;
		}{}
		
		\draw \bFront -- \bBack;
		\draw \cFront -- \cBack;
		\draw \dFront -- \dBack;
		}
	}

\ifthenelse{\equal{\sliceType}{lateral}}
	{
	\foreach\sliceCount in {0,...,\numSlicesMinusOne}
		{	
		\begin{scope}[shift ={(\sliceCount*\sliceWidth + \sliceCount*\sliceGap, 0)}]
			\tempTensor;
		\end{scope}
		}
	
	}
	{
	
	\ifthenelse{\equal{\sliceType}{frontal}}
	{
	
	\pgfmathsetmacro{\xStep}{\sliceDepth/2 + \sliceGap/2 + \myScale*\dimThree*\xSpec/(\nSlices-(1-\sWidth))}
	\pgfmathsetmacro{\yStep}{\sliceDepth/2 + \sliceGap/2 +  \myScale*\dimThree*\ySpec/(\nSlices-(1-\sWidth))}
	
	\foreach\sliceCount in {-\numSlicesMinusOne,...,0}
		{	
		
		\begin{scope}[shift = {(-\sliceCount*\xStep, -\sliceCount*\yStep)}]
			\tempTensor;
		\end{scope}
	
		}
	
	}
	{
	\tempTensor;
	}
	
	}

\node at (#2 + \dimTwo/2, #3 + \dimOne/2) {#4};

}} 

\headers{Structured Matrix Approximations via Tensor Decompositions}{M. E.\ Kilmer, A. K.\ Saibaba}

\title{Structured Matrix Approximations via Tensor Decompositions
\thanks{Submitted to the editors DATE.
\funding{This material is partly based upon work supported by the National Science
Foundation under Grants {NSF DMS} 1821148, 1821149, and
Tufts T-Tripods
Institute NSF HDR grant CCF-1934553. 
This material is partly based upon work supported by the National Science Foundation under Grant No. DMS-1439786 and by the Simons Foundation Grant No. 50736 while the first author was in residence at the Institute for Computational and Experimental Research in Mathematics in Providence, RI, during the ``Model and dimension reduction in uncertain and dynamic systems'' program.}}}
\author{Misha E.\ Kilmer\thanks{Department of Mathematics, Tufts University, Medford, MA 02155 ({\tt misha.kilmer@tufts.edu}),}
 \and Arvind K.\ Saibaba\thanks{Department of Mathematics, NC State University, Raleigh, NC
 ({\tt asaibab@ncsu.edu}),} 
}

\begin{document}

\maketitle

\begin{abstract}
We provide a computational framework for approximating a class of structured matrices; here, the term structure is very general, and may refer to a regular sparsity pattern (e.g., block-banded), or be more highly structured (e.g., symmetric block Toeplitz). The goal is to uncover {\it additional latent structure} that will in turn lead to computationally efficient algorithms when the new structured matrix approximations are employed in the place of the original operator. Our approach has three steps: map the structured matrix to tensors, use tensor compression algorithms, and map the compressed tensors back to obtain two different matrix representations---sum of Kronecker products and block low-rank format. The use of tensor decompositions enables us to uncover latent structure in the problem and leads to compressed representations of the original matrix that can be used efficiently in applications.
The resulting matrix approximations  are memory efficient, easy to compute with, and preserve the error that is due to the tensor compression in the Frobenius norm.  Our framework is quite general.  We illustrate the ability of our method to uncover block-low-rank format on structured matrices from two applications: system identification, space-time covariance matrices.  In addition, we demonstrate that our approach can uncover sum of structured Kronecker products structure on several matrices from the SuiteSparse collection.  Finally, we show that our framework is broad enough to encompass and improve on other related results from the literature, as we illustrate with the approximation of a three-dimensional blurring operator. 

\end{abstract}



\section{Introduction}
Structured matrices are ubiquitous in many applications such as signal processing, image deblurring, finite difference discretizations of partial differential equations (PDEs), and control theory.  For example, modeling a system in which the operator is spatially invariant in both the vertical and horizontal directions, with zero boundary conditions (e.g., blur in an image) leads, upon discretization, to a matrix that is block Toeplitz with Toeplitz blocks (BTTB).  Discretization of PDE models often leads to matrices that are at least block banded and perhaps possess additional structure. 
Matrices with such structure are advantageous for iterative linear solvers such as Krylov subspace methods because the required matrix-vector products per iteration can usually be done efficiently by leveraging the structure. 

However, many applications demand a factorization of the operator that reveals information about the rank or will allow for a low-rank approximation.  For large scale problems, the classes of structured matrices for which such decompositions are easily and directly computable are limited -- e.g. matrices that are diagonalized or block-diagonalized by fast transform based methods (e.g., BCCB), or matrices generated by a single Kronecker product.  One of the purposes of this
paper, therefore, is to expand the classes of structured matrices for which low-rank factorizations can be obtained.  Specifically, we give a framework for computing block low- rank format which is applicable to a general class of structured matrices. 

There are many applications in which the operator is expressed naturally as a sum of Kronecker products \cite{van2000ubiquitous}. 
If the number of the terms in the sum is small, the Kronecker structure can be leveraged to compute matrix-vector products quickly, or to develop preconditioners or approximate truncated SVD filters (e.g. \cite{GarveyMengNagy2018}). Storage of the relevant components needed to perform products with the operator or its transpose can be highly efficient.  

Our goal in this paper therefore is 
to compute operator approximations that discover either/both the Kronecker sum structure and block low-rank structure while preserving other underlying matrix structure, such as (block) Toeplitz or Hankel, (block) bandedness, symmetry and positive (semi)-definiteness.  
We propose to arrive at these operator approximations in an innovative way in three steps:  first, we remove redundancies through a matrix-to-tensor mapping, with the order of the tensor as determined by properties of the matrix; second, we utilize tensor-decomposition methods at our disposal to compute a tensor approximation; finally, we obtain our matrix approximation by mapping the tensor approximation back to an operator by inverting our matrix-to-tensor mapping. 

\paragraph{Contributions}
We summarize the contents of the paper and highlight the major contributions. We develop a tensor-based framework for computing structure-preserving matrix approximations for use in applications from scientific computing. 
\begin{enumerate}
    \item We develop a mapping between structured block matrices to tensors. The tensors can then be compressed and stored in the appropriate tensor format. We then relate  the error in tensor representation to the matrix representation and show that these errors are the same in the Frobenius norm.
    \item We use the CP and Tucker formats for compressing the tensor, and we show how they can be used for two different efficient matrix representations---Kronecker product sums, and block low-rank formats. Our methods are structure-preserving in the sense that both of these efficient matrix representations retain important features from the original structured block matrices. In particular, we pay special attention to preserving positive definiteness and semidefiniteness.
    \item We demonstrate the computational benefits of our approach on two different application areas: eigensystem realization algorithm for system identification and space-time covariance matrices, which yield block Hankel and (symmetric) block Toeplitz matrices respectively. We also demonstrate our approach on several sparse block tridiagonal test matrices from the SuiteSparse matrix collection. 
    \item We show how to extend this approach that we developed for structured block matrices to multilevel structured block matrices. We show how to compute efficient Kronecker product sums after compressing the tensor using Tucker format. 
    \item We discuss how our tensor-based approaches are related to and can recover existing algorithms from the literature and how we might use our approach to improve on some of them. 
\end{enumerate}

\paragraph{Relation to previous literature} The present approach differs from the methods in the literature on direct operator approximation by use of Tensor Train tensor approximations and related tensor formats~\cite{olshevsky2006tensor,oseledets2011tensor,oseledets2010approximation,oseledets2012solution}, which are particularly effective methods in generating compact, approximate representations to high dimensional operators in order to solve linear systems.  Our methodology and goals are different. We also mention the work in ~\cite{batselier2018computing} that has similar goals to ours but is different both in the way the tensors are constructed and compressed. 
Furthermore, in our approach, several types of structured approximations are possible, including Kronecker-sum, and block structured low-rank approximation.   
The appeal of our approach is the general nature of the term ``structure'' as well as the fact that it is agnostic to which type of tensor approximation algorithm is employed.  Our approach is flexible in that any tensor approximation method can be utilized in the second step, and we can provably relate the quality of our matrix approximation. As we discuss in \Cref{sec:multilevel}, the framework we provide here either subsumes or allows improvement upon some other methods in the literature that aim to provide optimal/near optimal Kronecker product approximations to  structured matrices, such as those in \cite{KammNagy,nagy2006kronecker}.

\paragraph{Organization and Contents} The first part of this paper is devoted to a general formulation of structure and developing a matrix-to-tensor invertible mapping that can utilize this formulation.  While we describe this in detail for two level matrices having structure on the outermost layer, we show that our method is applicable for higher level structure as well.  Since the matrix-to-tensor mapping is invertible, if we are given an approximation to the tensor representation of our matrix, we can map this approximation back to a matrix approximation. We give a result that relates the quality of the tensor approximation to the quality of the matrix approximation.   
Next, we investigate the specific structure that the matrix approximation can inherit depending on the type of tensor approximation used.  For this, we focus on a CP formulation of our tensor approximation, and on a truncated-HOSVD based approximation.

 This paper is organized as follows.  \Cref{sec:background} describes the notation and terminology, including a brief description of the CP and Tucker decompositions.  In \Cref{sec:mattotens}, we describe the block matrix to tensor mapping, with examples for clarity.  We also give a theoretical result that relates the tensor approximation problem to the matrix approximation problem. \Cref{sec:kron} describes how we obtain Kronecker sum approximations to the operator from CP or Tucker approximations to the tensor.  The subject of \Cref{sec:blockstructured} is how to obtain matrix approximations that are block structured factorization from the CP or Tucker approximations to the original tensor.   To this end, we also discuss how to preserve symmetric positive definiteness.  An application in subspace system identification and in space-time covariance matrices is given in \Cref{sec:apps}. We also process several matrices from the SparseSuite in \Cref{sec:apps}.   Our framework is generalized to capture higher level structure in \Cref{sec:multilevel}.  As
an illustration of the utility of our approach, we highlight how our framework can be used to capture or improve on a few previous results in the literature.  Conclusions and future work are addressed in \Cref{sec:final}.

\section{Background and Notation}  \label{sec:background}
We briefly clarify the notation that we use in this paper and some background material on Kronecker products and tensors. In this brief review, we mostly focus on third order tensors; however, they can be readily extended to higher order tensors and we refer to~\cite{kolda2009tensor,hackbusch2012tensor} for the details.

\subsection{Notation and Terminology}
Throughout this paper, matrices are defined in bold uppercase, vectors are in bold lowercase, and third order tensors are in calligraphic boldface script. 

Let $\TA \in \mathbb{R}^{m \times p \times n}$. We denote the frontal slices as $\TA_{:,:,j}$, lateral slices as $\TA_{:,j,:}$, and horizontal slices as $\TA_{j,:,:}$. 
Following \cite{kilmer2013third}, we use invertible mappings between $m \times n$ matrices and $m \times 1 \times n$ tensors 
by twisting and squeezing\footnote{If $\TA$ is $m \times 1 \times n$, the {\sc Matlab} command $\mbox{\tt squeeze}(\TA)$ returns the $m \times n$ matrix.  Note, however, that if the first mode has dimension 1, the squeeze command in {\sc Matlab} behaves differently, depending on the size of the other modes.  For example, the \mbox{\tt squeeze} of a $1 \times 1 \times n$ tensor  produces an length-$n$ column vector, the \mbox{\tt squeeze} of a $1 \times m\times n$ produces an $m\times n$ matrix.}: i.e., $\MX \in \mathbb{C}^{m \times n}$ is related to ${\TX}_{:,j,:}$ via
\[  {\TX} = \mbox{\tt twist}(\MX)  \mbox{ and }  \MX = \sq( {\TX}_{:,j,:}). \]

Tensor unfoldings refer to unwrapping a tensor into a matrix.  For example, a third order tensor $\TA\in \mb{R}^{m \times p \times n}$ admits three unfoldings \cite{kolda2009tensor}:
\[ \MA_{(1)} = [ \TA_{:,:,1}, \TA_{:,:,2}, \cdots, \TA_{:,:,n}]  \in \mathbb{C}^{m \times n p}\]
\[ \MA_{(2)} = [ \TA_{:,:,1}^T,\TA_{:,:,2}^\top, \cdots, \TA_{:,:,n}^\top ]  \in \mathbb{C}^{p \times m n} \]
and
\begin{equation} \label{eq:mode3unfold} \MA_{(3)} = [ \mbox{sq}(\TA_{:,1,:})^\top,\mbox{sq}(\TA_{:,2,:})^\top,\ldots,\mbox{sq}(\TA_{:,p,:})^\top] \in \mathbb{C}^{n \times mp }.\end{equation}
The unfoldings can be used to define mode-wise products between matrices and tensors.   For example, let 
\[ \TA \times_i \MU  \equiv  \MU \MA_{(i)}, \qquad i=1,2,3,\]
where it is clear that the column dimension of $\MU$ must match the row dimension of $\MA_{(i)}$ and the resulting tensor must have $i$th mode of dimension $r$, where $r$ is the number of rows in $\MU$. 

If $\V{a} \in \mb{R}^m, \V{b} \in \mb{R}^\ell, \V{c} \in \mb{R}^{n}$ , then 
\[ \TA := \V{a} \circ \V{b} \circ \V{c}  \in \mb{R}^{m \times \ell \times n}\]
is a third order, rank-1 tensor, with $(i,j,k)$ entry of $\TA$ given by $\V{a}_i \V{b}_j \V{c}_k$.

For convenience, we may also use the {\tt vec} and {\tt reshape} commands to define mappings between matrices and vectors by column unwrapping and reshaping:
\[ {\V{a}}  = \mbox{\tt vec}(\MA) \in \mathbb{C}^{mn}  \leftrightarrow \MA = \mbox{\tt reshape}({\V{a}},[m,n]) .\]

\subsection{Kronecker products}

As we will see, the recovered matrix approximations all involve one or more Kronecker products.  Let $\MA \in \mb{R}^{m \times p}$ and $\MB \in \mb{R}^{n \times \ell}$. Then their Kronecker product $\MB \otimes \MA\in \mb{R}^{(mn)\times (p \ell)}$ is the block matrix 
\[ \MB \otimes \MA = \begin{bmatrix} b_{11} \MA & b_{12} \MA & \cdots & b_{1p} \MA \\
                                                       b_{21} \MA & b_{22} \MA & \ddots & \vdots \\
                                                         \vdots & \cdots  & \cdots & \vdots \\
                                                       b_{m1} \MA & \cdots & \cdots & b_{mn} \MA  \end{bmatrix}. \]

Kronecker products satisfy some important properties \cite{van2000ubiquitous} which can be used to great computational advantage in applications.  We review a few of these properties: 
 \begin{enumerate}
   \item $\mbox{\tt vec}(\MA \MX \MB^\top) \equiv (\MB \otimes \MA) \mbox{\tt vec}(\MX)$, where the $\mbox{\tt vec}(\cdot)$ is a vectorization operator converts a matrix to a column vector by stacking its columns. 
   \item For $\MA, \MB, \M{C},\MD$ of appropriate dimension $(\MA \otimes \MB)(\M{C} \otimes \MD) = (\MA \M{C}) \otimes (\MB \MD)$. 
   \item Kronecker products of orthogonal matrices are orthogonal.  
   \item If $\MA$ is a structured matrix (e.g., diagonal, banded, upper triangular, Toeplitz, etc), then $\MA \otimes \MB$ will inherit that structure on the block level.   
   \item The SVD of $\MA \otimes \MB$, is obtained from the SVDs $\MA = \MW \MS \MQ^\top$ and $\MB = \MU \MD \MV^\top$ since $(\MU \MD \MV^\top) \otimes (\MW \MS \MQ^\top) = (\MU \otimes \MW) (\MD \otimes \MS) (\MV \otimes \MW)^\top$.  Note
that though the singular values of $\MA \otimes \MB$ are contained in the diagonals of the diagonal matrix $\MD \otimes \MS$, they are likely not
in order, so the ordered SVD would require an additional permutation. 
\item As a consequence of the last point, $\|\MA \otimes \MB \|_F = \|\MA\|_F\|\MB\|_F$.
  \end{enumerate}

\subsection{CP Decomposition}\label{ssec:cp}

Suppose the tensor $\TA \in \mathbb{R}^{m \times p \times n}$ can be expressed
\[ \TA = \sum_{i=1}^r \MX_{:,i} \circ \M{Y}_{:,i} \circ \M{Z}_{:,i}, \]
where $\MX \in \mathbb{R}^{m \times r}$, $\M{Y} \in \mathbb{R}^{p \times r}$, $\M{Z} \in \mathbb{R}^{n \times r}$ are called the {\it factor matrices}.   The representation is also expressed using the so-called Kruskal notation in terms of the factor matrices $ \llbracket \MX,\MY, \M{Z} \rrbracket $.  
Such a decomposition is called the CP\footnote{Also goes by CANDECOMP/PARAFAC, Canonical Factors, etc.} decomposition of a tensor.   If $r$ is minimal, then $r$ is called the rank of the tensor.   
Unfortunately, determining the rank of a tensor is in general an NP-hard problem.  Moreover, the factor matrices do not need to have independent, let alone orthogonal, columns.  For information about properties of the CP decomposition, see \cite{kolda2009tensor}.

Independent of whether or not $r$ is minimal, if there exists a triple of factor matrices such that $\TA = \llbracket \MX, \MY, \M{Z} \rrbracket$, then we can use the factor matrices to deduce relationships among the various slices of the tensor.  For example, the {$k$th lateral slice} has entries corresponding to the triple matrix product
\begin{equation} \label{eq:cplat} \sq (\TA_{:,k,:}) = \MX \mbox{diag}(\M{Y}_{k,:}) \M{Z}^\top , \qquad k=1,\dots,p, \end{equation}
with similar formulations for the other slices. 
Thus, a CP decomposition of a third order tensor reveals a `joint diagonalization' of the matrices that comprise the slices of 
the tensor, although 
the matrices in the diagonalization {need not be square nor invertible}. A popular way of computing the CP decomposition is the alternating least squares approach; see e.g.,~\cite[Section 3.4]{kolda2009tensor}.

\subsection{Tucker Decomposition}\label{ssec:hosvd}
The Tucker-3 factorization \cite{tucker1966some} of a third-order tensor $\TA \in \mb{R}^{m\times p \times n}$ is 
\[  \TA = \T{G} \times_1 \MU \times_2 \MV \times_3 \MW = \sum_{i=1}^{R_1} \sum_{j=1}^{R_2} \sum_{k=1}^{R_3} 
\T{G}_{i,j,k} ( \MU_{:,i} \circ \MV_{:,j} \circ \MW_{:,k}) , \]
where $\T{G} \in \mb{R}^{R_1 \times R_2 \times R_3}$ is the core tensor, $\MU\in \mb{R}^{m \times R_1}$, $\MV\in \mb{R}^{p \times R_2}$, $\MW\in \mb{R}^{n \times R_3}$.  The matrices $\MU, \MV, \MW$ are also called factor matrices. We say that the tensor is rank$-(R_1,R_2,R_3)$, if $R_j = \text{rank}(\MA_{(j)})$ for $j=1,\dots,3$.   The core is not typically `diagonal' and the elements need not be non-negative.  If the core is diagonal, and $R_1 = R_2 = R_3$ then this decomposition reduces to a CP decomposition.

For later use, we highlight one important feature that relates the mode-3 unfolding of $\TA$ with the mode-3 unfolding of the core $\TC$ and the factor matrices~\cite{kolda2009tensor}:
\begin{equation} \label{eq:mode3flatkron}  
      \TA_{(3)} = \MW \T{G}_{(3)} (\MV^\top \otimes \MU^\top)  .
\end{equation}
Combining~\eqref{eq:mode3flatkron} with~\eqref{eq:mode3unfold}, we observe that \textbf{$k$th lateral slice} of $\TA$
is given by 
\begin{equation} \label{eq:klatslice}  
   \sq(\TA_{:,k,:}) = \MU \left(\sum_{j=1}^{r_2} v_{k,j} \sq(\T{G}_{:,j,:})\right) \MW^\top \qquad k=1,\dots,p. 
  \end{equation}

Of particular interest is the Higher Order SVD (HOSVD)~\cite{de2000multilinear} for obtaining a compressed representation in the Tucker format. Suppose we want to compress $\TA \in \mb{R}^{m\times p \times n}$ and the target rank is $(r_1,r_2,r_3)$.  The HOSVD first proceeds by computing  the SVDs of the 3 unfoldings and letting $\MU, \MV, \MW$ correspond to the top $r_1,r_2,r_3$ left singular vectors of each mode unfolding. Next, the core tensor is obtained as $\T{G} = \TA \times_1 \MU^\top \times_2 \MV^\top \times_3 \MW^\top$, so that we have the approximation
\[ \TA \approx \T{G} \times_1 \MU \times_2 \MV \times_3 \MW.\]
For short, we write $\TA \approx [\TG; \M{U}, \M{V}, \M{W}]$. The accuracy of the HOSVD can be established by the results in~\cite[Thoerem 10.3]{hackbusch2012tensor}. Besides the HOSVD, there are other efficient algorithms for approximating a tensor in the Tucker format. For some recent work on randomized algorithms in the Tucker format, see~\cite{minster2020efficient,sun2020low}.


\section{Block Matrix to Tensor Mapping}\label{sec:mattotens} 

In this section, we consider a block matrix $\M{A} \in \mb{R}^{(\ell m) \times (q n)}$ with $\ell \times q$ blocks of size $m\times n$ each. The basic idea is to first represent the matrix $\M{A}$ using a sum of Kronecker products.  Let $\M{A}_k \in \mb{R}^{m \times n}$ be a {\it distinct} subblock of a matrix which repeats $\eta_k > 0$ times for $k=1,\dots,p$. Here $1 \leq p \leq \ell q$ is the number of distinct subblocks. Define the tuple of matrices\footnote{This tuple is unique up to permutation.} 
\begin{equation*}
   \mc{A} =  \> (\M{A}_1,\dots,\M{A}_p ) \qquad 
        \mc{E} =  \> (\M{E}_1,\dots,\M{E}_p ),
 \end{equation*}
where $\M{E}_k \in \mb{R}^{\ell\times q}$ with entries 
\[ [\M{E}_k]_{ij} = \left\{ \begin{array}{ll} \frac{1}{\sqrt{\eta_k}} & \text{if } \M{A}_k \text{ occurs in } (i,j)^\text{th} \text{ block of }\M{A}   \\ 0 & \text{otherwise}. \end{array} \right.\]
 The matrices $\M{E}_k$ both account for the position and repetition of the blocks $\M{A}_k$ and satisfy \begin{equation} \label{eq:fnormw} \|\M{E}_k\|_F = 1, \qquad k=1,\dots,p. \end{equation} 
 To express the original matrix $\M{A}$ as the sum of Kronecker products, we define
\begin{equation}
    \label{eqn:kronsum}
    \struct (\mc{A}) = \sum_{k=1}^p \M{E}_k \otimes  \left( \sqrt{\eta_k} \,\M{A}_k \right).
\end{equation}
Then it is verified that $\struct(\mc{A}) = \M{A}$. 

\begin{remark}
Consider the trace inner product $\langle \M{C},\M{D} \rangle = \trace(\M{C}^\top\M{D})$ for matrices $\M{C},\M{D}$ of the same size.  If we let $\M{C} = \M{E}_j \otimes \M{K} $ and $\M{D} = \M{E}_k \otimes   \M{L}$ for $\M{K},\M{L} \in \mb{R}^{m\times n}$ and $j\neq k$, then $\M{C}$ and $\M{D}$ are orthogonal with respect to this inner product, since
 \[ \langle \M{C},\M{D} \rangle =\trace(\M{E}_j^\top\M{E}_k) \trace(\M{K}^\top\M{L}) = 0.  \]
This means that in~\eqref{eqn:kronsum} we have decomposed $\MA$ into a sum of orthogonal matrices with respect to the trace inner product.
\end{remark}

We illustrate this data structure to represent through several commonly occurring examples. For simplicity, we take the number of blocks $\ell = q$, but the general case can be handled easily.
\begin{example}[Block-diagonal matrix] We consider a block-diagonal matrix of the form
    \[ \M{A} = \bea \M{A}_1 \\ & \M{A}_2 \\ & &  \ddots \\ & & & \M{A}_p \eea \in \mb{R}^{m\ell \times n\ell}. \]
    In this case we have $p = \ell = q$, and
    \[ \mc{A} = ( \M{A}_1,\dots,\M{A}_p )  \qquad \mc{E} = (\V{e}_1\V{e}_1^\top,\dots,\V{e}_p\V{e}_p^\top),\]
    where $\V{e}_j$ is the $j-$th column of a $p\times p$ identity matrix.
    
    \end{example}
  \begin{example}[Block-banded matrix] If the matrix is block tridiagonal with $\ell \times \ell $ blocks, then $p = 3\ell -2$ (block-non-symmetric) and $p = 2\ell -1$ (block-symmetric). More generally, if the matrix is block-banded and we denote by $b$ the block-semibandwidth, then $p = m(2b+1)-b(b+1)$ if the matrix is block-non-symmetric, and $p=m(b+1)-b(b+1)/2$ if block-symmetric. For example, consider a block-tridiagonal matrix of the form 
 \[ \M{A} = \bea \MA_1 & \MA_{2} & 0  & 0 \\ \MA_2 & \MA_3 & \MA_{4} & 0 \\ 0 & \MA_4 & \MA_5 & \MA_6 \\ 
                      0 & 0 & \MA_6 & \MA_{7} \eea \in \mb{R}^{4m\times 4n}.\]
    In this case we have $p= 7$, 
    \[ \mc{A} = ( \M{A}_1,\dots,\M{A}_{7} )  \qquad \mc{E} = ( \V{e}_1\V{e}_1^\top, \V{e}_2\V{e}_1^\top, \dots, \V{e}_3\V{e}_4^\top, \V{e}_4\V{e}_4^\top),\]
    where $\V{e}_j$ is the $j-$th column of a $4\times 4$ identity matrix.  
    \end{example}
    
\begin{example}[Block-Toeplitz]\label{ex:btoep} Consider the block-Toeplitz matrix of the form 

    \[\M{A} = \bea \MA_1 & \MA_{\ell+1} & \dots & \MA_{2\ell-1} \\ \MA_2 & \MA_1 & \dots & \MA_{2\ell-2} \\ \vdots & \ddots & \ddots & \vdots \\ \MA_\ell & \dots & \MA_2 & \MA_1 \eea \in \mb{R}^{\ell m\times \ell n}.\]
    
    Define the shift matrices  
    \[ \M{E}_d = \bea 0 & \\ 1 &  0\\  & \ddots & \ddots\\  & & 1 &0  \eea \in \mb{R}^{\ell\times \ell} \] 
    and $\M{E}_u = \M{E}_d^\top$. Then 
    \[  \mc{A} = ( \M{A}_1,\dots,\M{A}_{2\ell -1} ) \] 
    and 
    \[ \mc{E} = \left(  \frac{1}{\sqrt{\ell}}\M{E}_d^0, \frac{1}{\sqrt{\ell-1}}\M{E}_d, \dots,  \M{E}_d^{\ell-1},  \frac{1}{\sqrt{\ell-1}} \M{E}_u, \frac{1}{\sqrt{\ell-2}}\M{E}_u, \dots,\M{E}_u^{\ell-1} \right).\] 
    If the matrix is block-symmetric then additional savings are possible since $p = \ell$. Similarly, if the matrix is additionally block banded, then further savings are possible. A similar approach is possible if the matrix is block-Hankel instead of block-Toeplitz.
\end{example}

 In each case above, the reader is invited to verify that~\eqref{eqn:kronsum} holds. This representation is unique up to reordering of the terms in the summation.

Usually, knowledge of the problem structure can be used to find this representation. If no knowledge is available, and a purely algebraic technique is desired, this representation can be determined by a greedy search. Assuming that it takes $\mc{O}(1) $ operations to verify if the blocks are equal, this representation can be computed in $\mc{O}((\ell q)^2)$ operations (worst case complexity), but the complexity is much lower if there is latent structure (e.g., block Toeplitz, block-diagonal, etc) to be discovered. Furthermore, if the block sizes are large it may be advantageous to store the matrices in $\mc{E}$ in compressed sparse column representation.

With this matrix $\M{A} \in \mb{R}^{(\ell m) \times (q n)} $, and associated data structures $\mc{A}$ and $\mc{E}$, we associate a third-order tensor mapping  $\cTT_{\mc{E}} : \mb{R}^{\ell m \times qn} \rightarrow  \mb{R}^{m\times p \times n} $ defined in terms of its lateral slices as 
$$\cTT_{\mc{E}}[\M{A}]_{:,k,:} = \sqrt{\eta_k} \,\mathsf{twist}(\M{A}_k)  \qquad k = 1,\dots p.$$ 
The tensor-to-matrix mapping $\mc{M}_{\mc{E}}:\mb{R}^{m\times p\times n} \rightarrow \mb{R}^{\ell m\times qn}$ is defined as 
\[ \mc{M}_{\mc{E}}[\TX] := \struct\left(\mathsf{sq}(\TX_{:,k,:}) \right)_{k=1}^p = \sum_{k=1}^p\M{E}_k \otimes \mathsf{sq}(\TX_{:,k,:}).\]
It is easy to verify that $\mc{M}_\mc{E}[\cTT_\mc{E}[\M{A}]] = \M{A}$ since
\[ \mc{M}_\mc{E}[\cTT_\mc{E}[\M{A}]] = \sum_{k=1}^p\M{E}_k \otimes \sq(\sqrt{\eta_k}\twist(\M{A}_k)) = \sum_{k=1}^p\M{E}_k \otimes \left(\sqrt{\eta_k} \, \M{A}_k \right) = \M{A}. \]

\subsection{Strategy}
Given $\MA$ and the new mapping scheme, we are now ready to present our approach to generating our matrix approximations.  The steps are outlined in
\Cref{alg:strategy}.  It is important to note that
this approach is applicable for any choice of tensor decomposition and corresponding tensor approximation
strategy.  However, the practical utility is tied to the structure obtained after applying the tensor-to-matrix mapping in the last step.  Thus, in \Cref{sec:kron}, we consider some specific tensor approximation strategies and delve into the resulting matrix structure.  
\begin{algorithm}[ht]  \caption{\label{alg:strategy} Tensor-based Matrix Approximation} 
\begin{algorithmic}[1]
\STATE Determine $\mc{E}, \mc{A}$ from $\MA$. \\
\STATE Define $\TX:= \cTT_{\mc{E}}[\MA]$. \\
\STATE Compute a tensor approximation to $\TX$, $\TTT := \widehat{\cTT}_\mc{E}[\MA]$. \\
\STATE Define the matrix approximation $\widehat{\MA}:=\mc{M}_\mc{E}[{\TTT}]=\mc{M}_\mc{E}[\widehat{\cTT}_\mc{E}[\MA]]$. 
\end{algorithmic}
\end{algorithm}

For our approach to be successful, we need to know how to control the error in the matrix approximation.  Fortunately, we can directly connect the error in the tensor approximation with the error in the matrix approximation, as we now show.  
This result is general in that does not depend on the specific tensor structure that is used, or the algorithm used for low-tensor rank approximation. 
\begin{lemma}\label{lem:tenapprox} Let $\M{A} \in \mb{R}^{(\ell m) \times (q n)} $ and let $\cTT_{\mc{E}}[\cdot]$ and $\mc{M}_{\mc{E}}[\cdot]$ be the tensor-to-matrix and matrix-to-tensor mappings respectively. Let $\widehat{\cTT}_\mc{E}[\M{A}] \approx {\cTT}_\mc{E}[\M{A}]$ be a tensor approximation computed using any appropriate method. Then the error in the matrix approximation satisfies
\[\|\M{A} -\mc{M}_\mc{E}[\widehat{\cTT}_\mc{E}[\M{A}]] \|_F  =  \|\cTT_{\mc{E}}[\M{A}] - \widehat{\cTT}_{\mc{E}}[\M{A}]\|_F.\]
\end{lemma}
\begin{proof}
For simplicity, write $\TX = \cTT_{\mc{E}}[\M{A}]$ and ${\TTT} = \widehat{\cTT}_{\mc{E}}[\M{A}]$. By the assumptions made earlier, there are $p$ distinct, and non-overlapping blocks and each block is repeated $\eta_k$ times.  Combining this insight with the properties of Kronecker products and Frobenius norms:
\[ \begin{aligned}
\|\M{A} - \mc{M}_\mc{E}[{\TT}]\|_F^2 = & \> \left\|\sum_{k = 1}^p\M{E}_k \otimes \sq(\TX_{:k:}) - \sum_{k = 1}^p\M{E}_k  \otimes \sq({\TT}_{:k:})\right\|_F^2\\
= & \> \sum_{k = 1}^p \| \M{E}_k\otimes (\sq(\TX_{:k:}) - \sq({\TT}_{:k:})) \|_F^2 \\
=&  \> \sum_{k = 1}^p \|\M{E}_k \|_F^2\| \TX_{:k:} - {\TT}_{:k:} \|_F^2.
\end{aligned} \]
Taking square roots, we have the desired result.
\end{proof}

\begin{remark} In the worst case, if there no block-structure, we have $p =\ell q$. In this case, we can write 
\[ \M{A} = \sum_{j=1}^\ell \sum_{k = 1}^q \M{E}_{jk}\otimes \M{A}_{jk}, \]
where $\M{E}_{jk} = \V{e}_j\V{e}_k^\top \in \mb{R}^{\ell \times q}$, and $\M{A}_{jk} \in \mb{R}^{m\times n}$ is the $(j,k)$th block matrix of $\M{A}$. In this case, it may be more appropriate to define the tensor mapping $\cTT_{\mc{E}}[\M{A}] \in \mb{R}^{m\times \ell \times q \times n}$, with lateral slices 
\[ \cTT_{\mc{E}}[\M{A}]_{:,j,k,:} = \twist(\M{A}_{jk}) .\]
The important point is that it may be more appropriate to represent it as a 4th order tensor, rather than a 3rd order tensor. We will not consider this case here and leave it for future work. 
\end{remark}
 
\begin{remark} The result in~\cref{lem:tenapprox} is independent of the tensor approximation method that is used.  In the next sections, we describe the specific structure that can be obtained when using CP or truncated-HOSVD expansions, but it is important to note that other options are possible, and we do not cover them all here. \end{remark}

\begin{table} \centering
\caption{Summary of the notation used to map matrices to tensors.}
\label{tab:summary}
    \begin{tabular}{c|c|c}
    Symbol & Dimensions & Description \\ \hline
   
    $\T{X} = \mc{T}_{\mc{E}}[\M{A}]$& $m \times p \times n$  & Matrix mapped to tensor\\
     $\M{A} ={\mc{M}}_{\mc{E}}[\T{X}] $ & $(\ell m)\times (qn)$ & Tensor mapped to matrix  \\
    $\TT = \widehat{\mc{T}}_{\mc{E}}[\M{A}]$& $m \times p \times n$  & Tensor approximation \\
     $\widehat{\M{A}} = {\mc{M}}_{\mc{E}}[\TT]$& $(\ell m)\times (qn)$  & Matrix approximation
    \end{tabular}
\end{table}



\section{Kronecker sum approximations using tensor decompositions}\label{sec:kron}
In the notation of the previous section, suppose we have a tensor ${\cTT}_\mc{E}[\M{A}]$ that is the result of mapping a structured matrix $\M{A}$.   We decompose (approximately) the tensor to obtain $\widehat{\cTT}_\mc{E}[\M{A}]$, then we use the inverse mapping to get our matrix approximation. For simplicity, we write $\TX:= \cTT_{\mc{E}}[\MA]$ and $\TTT := \widehat{\cTT}_\mc{E}[\MA]$.  We will show that when the tensor approximations are available in either a CP or tr-HOSVD format, we obtain matrix approximations that involve a level of Kronecker structure.  Indeed, using the right lens, we can express our matrix approximation as a sum of (structured) Kronecker products of matrices.   

\begin{figure}[!ht]\centering
\begin{tikzpicture}[scale=0.3]
 \tikzmath{\h = 3; \w = 1;}
\tikzmath{\l = 0; \t = 3*\h; }
\foreach \x in {0,\h,2*\h}{
\filldraw[fill=blue, draw=black] (\l+ \x, \t - \x -\h ) rectangle ++(\h,\h);
}
\foreach \x in {\h,2*\h}{
\filldraw[fill=blue!50, draw=black] (\l+ \x, \t - \x) rectangle ++(\h,\h);
\filldraw[fill=blue!50, draw=black] (\l+ \x-\h , \t - \x -\h) rectangle ++(\h,\h);
}

\filldraw[fill=blue!20, draw=black] (\l+2*\h, \t -1*\h) rectangle ++(\h,\h);
\filldraw[fill=blue!20, draw=black] (\l, \t -3*\h) rectangle ++(\h,\h);

 \tikzmath{\l = \l +  5*\h; \t = 2*\h;}
\draw (\l -1.5*\h, 1.5*\h) node[fill=none,scale=2]   {$\approx$};  
\draw (\l -\w, 1.5*\h) node[fill=none,scale=2]   {$\sum_j$};  
\foreach \x in {1,2,3}{
\filldraw[fill=blue, draw=black] (\l+ \x, \t - \x) rectangle ++(\w,\w);
}

\foreach \x in {1,2}{
\filldraw[fill=blue!50, draw=black] (\l+ \x + \w, \t - \x) rectangle ++(\w,\w);
\filldraw[fill=blue!50, draw=black] (\l+ \x , \t - \x -1*\w) rectangle ++(\w,\w);
}

\filldraw[fill=blue!20, draw=black] (\l+3*\w, \t -1*\w) rectangle ++(\w,\w);
\filldraw[fill=blue!20, draw=black] (\l+1*\w, \t -3*\w) rectangle ++(\w,\w);

\draw (\l + 2.5*\w,\t -4*\w) node[fill=none,scale=1.5] {$\mathbf{C}_j$};

\tikzmath{\l = \l +  7*\w; \t = 2*\h;}
\draw (\l-0.3*\w,\t-\h-0.3*\w) -- (\l-0.5*\w,\t- \h  -.3*\w) --  (\l-0.5*\w,\t+0.3*\w) -- (\l-0.3*\w,\t+0.3*\w);
\draw (\l -2*\w, 1.5*\h) node[fill=none,scale=1.5]   {$\otimes$};  
\filldraw[fill=gray!40, draw=black] (\l , \t -\h ) rectangle ++(\w,\h);
\filldraw[fill=gray!40, draw=black] (\l + 1.5*\w, \t -\w ) rectangle ++(\w,\w);
\filldraw[fill=gray!40, draw=black] (\l + 3*\w  , \t -\w ) rectangle ++(\h,\w);

\draw (\l + 6.3*\w,\t-\h-0.3*\w) -- (\l+6.5*\w,\t- \h  -.3*\w) --  (\l+6.5*\w,\t+0.3*\w) -- (\l+6.3*\w,\t+0.3*\w);
\draw (\l + 2*\w,\t - 4*\w) node[fill=none,scale=1.5] {$\mathbf{D}_j$};

\end{tikzpicture}

\caption{A visualization of the Kronecker sum approximation of the structured matrix.}
\label{fig:kroneckersum}
\end{figure}
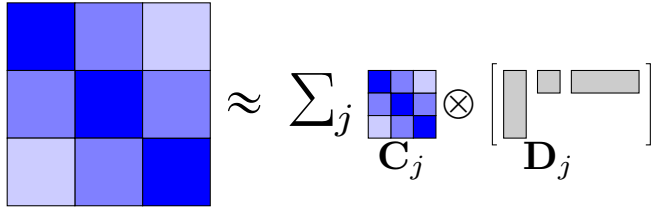
\subsection{Matrix Structure using CP}

Suppose we are given the approximation to $\TX\in\mb{R}^{m \times p \times n}$  in the Kruskal form for a third-order tensor $\TTT  = \llbracket \M{X}, \M{Y}, \M{Z} \rrbracket$.  We observe the following for the corresponding tensor-to-matrix mapping: 
\begin{theorem} \label{th:tildeA} Let $\MA \in \mb{R}^{(\ell m)\times (qn)}$ and let $\T{X} := \cTT_{\mc{E}}[\MA]$ be the mapped tensor. Suppose we approximate $\T{X}$ using $ {\TTT} = \llbracket \M{X}, \M{Y}, \M{Z} \rrbracket $, where $\M{X} \in \mb{R}^{m\times r}$, $\M{Y} \in \mb{R}^{p\times r}$,
$\M{Z} \in \mb{R}^{n\times r}$. Then 
\begin{equation} \label{eq:mystruct}  
 \cM_{\mc{E}}[ {\TTT} ] = ( \MI_\ell \otimes \M{X}) \struct\left( \diag(\MY_{k,:}) \right)_{k=1}^p  (\MI_q \otimes \M{Z}^\top) .\end{equation}
The term in the middle of the factorization is block $\ell \times q$ with $r \times r$ diagonal blocks.
\end{theorem}

\begin{proof}
 The Kruskal form of the tensor means that each lateral slice of $\TTT$ must have the form $\M{X} \diag(\M{Y}_{k,:} ) \M{Z}^\top$. Therefore, since 
 \[
 \begin{aligned}
 \mc{M}_{\mc{E}}[{\TTT}] = & \> \sum_{k=1}^p \M{E}_k\otimes \M{X} \diag(\M{Y}_{k,:} ) \M{Z}^\top \\
 = & \> \sum_{k=1}^p (\M{I}_\ell\otimes \M{X})\left(\M{E}_k\otimes  \diag(\M{Y}_{k,:} )\right) (\M{I}_q \otimes \M{Z}^\top) \\
 = & \> ( \MI_\ell \otimes \M{X}) \left( \sum_{k=1}^p \M{E}_k\otimes  \diag(\M{Y}_{k,:})\right) (\MI_q \otimes \M{Z}^\top).
 \end{aligned}
 \] 
 The proof is completed by recognizing the middle term as a block-structured matrix.
\end{proof}

\begin{corollary} \label{corr:kronsumcp}
Given the assumptions and notation in Theorem \ref{th:tildeA}, suppose we can factorize $\MY = \MF \MG^\top$ with $\MF \in\mb{R}^{p\times r}$ and $\M{G} \in \mb{R}^{r\times r}$, then 
 \begin{equation}  \label{eq:kron1} \mc{M}_{\mc{E}}[{\TTT}] = \sum_{j=1}^r \struct \left(\MF_{:,j}\right)_{k=1}^p \otimes \left( \MX \diag (\M{G}_{:,j}) \MZ^\top \right), \end{equation}
 \end{corollary}
is a representation as a sum of Kronecker products of matrices.
\begin{proof}
We have $\diag(\M{Y}_{k,:}) = \diag( \sum_{j=1}^r f_{kj} \MG_{:,j}^\top) = \sum_{j=1}^r f_{kj} \diag( \MG_{:,j} ) $.
Hence, 
$$\struct\left( \diag (\MY_{k,:})\right)_{k=1}^{p}  = \sum_{j=1}^r \struct\left(\MF_{:,j}\right)_{k=1}^p \otimes \diag(\MG_{:,j}).$$    
 Putting this last line together with \cref{eq:mystruct} gives \cref{eq:kron1}.

\end{proof}

The important point about \cref{eq:kron1} is that we now have
a matrix approximation $\widehat{\MA} \approx \M{A}$ as follows 
\[ \widehat{\MA} = \sum_{j=1}^r \M{C}_j \otimes \M{D}_j, \]
where $\M{C}_j = \struct \left(\MF_{:,j}\right)_{j=1}^p$ and $\M{D}_j = \MX \diag (\M{G}_{:,j}) \MZ^\top$ for $j=1,\dots,r$. This decomposition into a sum of
Kronecker products of matrices, is all the more noteworthy since the $\M{C}_j$ have the same
{\it  block structure} as the original matrix.  See
\cref{fig:kroneckersum} for an illustration.

\begin{remark} If additional structure is imposed on the CP decomposition, this could result in additional structure for the corresponding matrix approximation.  For example, if $\MA$ is non-negative, so that the tensor is non-negative, the CP factors should be non-negative, implying the approximation will also have non-negative entries.  
\end{remark}

\subsection{Matrix structure using Tucker}
Recall that we denote tensor $\TX := {\cTT}_\mc{E}[\M{A}] \in \mb{R}^{m \times p \times n}$ is approximated by the tensor $ \TTT  := \widehat{\cTT}_\mc{E}[\M{A}]$  such that
\[\TTT = \T{G} \times_1 \MU \times_2 \MV \times_3 \MW,\]
with core tensor $\T{G} \in\mb{R}^{r_1\times r_2\times r_3}$
and the factor matrices
$\MU \in \mb{R}^{m \times r_1}$, $\MV \in \mb{R}^{p \times r_2}$, and $\MW \in \mb{R}^{n \times r_3}$ typically have orthonormal columns. There are several methods for computing low-rank approximations in the Tucker format. Some options include Higher Order SVD (HOSVD), Sequentially Truncated HOSVD (ST-HOSVD),  and Higher Order Orthogonal Iteration (HOOI). See~\cite{kolda2009tensor,hackbusch2012tensor} for detailed reviews of the various techniques. For large-scale tensors, many randomized variants have been devised in recent years~\cite{che2019randomized,minster2020efficient,ahmadi2020randomized} to reduce the computational cost. In this next result, we show how to convert the approximation in the Tucker form to obtain an approximation as the sum of Kronecker products.
\begin{theorem}\label{thm:tuckerkron} Consider the setting as in \cref{th:tildeA} but with the approximation in the Tucker format $\TTT =  \widehat{\cTT}_\mc{E}[\M{A}] = \T{G} \times_1 \MU \times_2 \MV \times_3 \MW$, as defined above. Then
\begin{eqnarray} \label{eq:structTucker}
  \mc{M}_{\mc{E}}[\TTT] & = &  \sum_{j=1}^{r_2} \struct\left(v_{kj}\right)_{k=1}^p \otimes (\MU \sq(\T{G}_{:,j,:}) \MW^\top).            
 \end{eqnarray}
\end{theorem}
\begin{proof}
From~\cref{eq:mode3flatkron}, the mode-3 unfolding of $\TTT$ has the expression $\TTT_{(3)} = \MU \T{G}_{(3)} (\MV^\top \otimes \MW^\top)$. It follows that each lateral slice of $\TTT$ is (see also~\cref{eq:klatslice})
\[ \sq(\TTT_{:,k,:}) = \MU \left( \sum_{j=1}^{r_2} v_{kj} \sq(\T{G}_{:,j,:} ) \right) \MW^\top .\]
Applying the definition of the matrix-to-tensor mapping
\begin{eqnarray*}
  \mc{M}_{\mc{E}}[\widehat{\cTT}_\mc{E}[\M{A}]] & = &        \struct \left(\sq(\TTT_{:,k,:}) \right)_{k=1}^p   = \sum_{k=1}^p \M{E}_k \otimes \sq(\TTT_{:,k,:}) \\
 & = & \sum_{k=1}^p \M{E}_k \otimes \left(\sum_{j=1}^{r_2} \struct\left(v_{kj}\right)_{k=1}^p \otimes (\MU \sq(\T{G}_{:,j,:}) \MW^\top) \right).
 \end{eqnarray*}
 Interchanging the order of the summations, we have 
 \begin{eqnarray*}
 \mc{M}_{\mc{E}}[\widehat{\cTT}_\mc{E}[\M{A}]] & = & \sum_{j=1}^{r_2} \sum_{k=1}^p \M{E}_k \otimes v_{kj} (\MU \sq(\T{G}_{:,j,:}) \MW^\top) \\ 
 & = & \sum_{j=1}^{r_2} \left(\sum_{k=1}^p \M{E}_k \otimes v_{kj}\right) \otimes  (\MU \sq(\T{G}_{:,j,:}) \MW^\top).
 \end{eqnarray*}
 The proof is complete by identifying $ \sum_{k=1}^p \M{E}_k \otimes v_{kj} = \struct \left( v_{kj}\right)_{k=1}^p$.
\end{proof}
Using this theorem, we can compute the Kronecker product representation
 \[\MA \approx \sum_{j=1}^{r_2}\M{C}_j \otimes \M{D}_j, \]
 where $ \M{C}_j = \struct \left(v_{kj}\right)_{k=1}^p$ and $\M{D}_j = \MU \sq(\T{G}_{:,j,:}) \MW^\top$ for $j=1,\dots,r_2$. It is interesting to note that the matrices $\M{C}_j$ preserve the original block structure of the matrix; for example, if $\M{A}$ is block-Toeplitz, then the matrices $\{\M{C}_j\}_{j=1}^{r_2}$ are Toeplitz. This is similar to what we observed in the CP case (see \cref{fig:kroneckersum}). 

We can take this analogy a step further. The CP decomposition $\TTT = \llbracket\M{X},\M{Y},\M{Z} \rrbracket$ can be written in Tucker format as $\TTT = [\T{I}; \M{X},\M{Y},\M{Z}]$ where $\T{I}$ is the identity tensor with $\T{I}_{ijk} = 1$ if $i=j=k$ else $0$. Note that in writing in the Tucker format we are not enforcing that the factor matrices have orthonormal columns. Applying \cref{thm:tuckerkron} with the Tucker form of $\TTT$ we obtain a result similar to \cref{corr:kronsumcp}.

\section{Block structured factorizations}\label{sec:blockstructured} 
In the previous section, we approximated $\M{A}$ as a sum of structured Kronecker products. In this subsection, we show how the tensor structure can be used to derive block-structured low-rank factorizations. Previous work in~\cite{savas2011clustered,savas2016clustered,wang2019block,amestoy2017complexity} also considered block low-rank factorizations but our work is novel in two different ways: we consider the additional structure present in the matrix and we leverage tensor-based algorithms to exploit this structure.

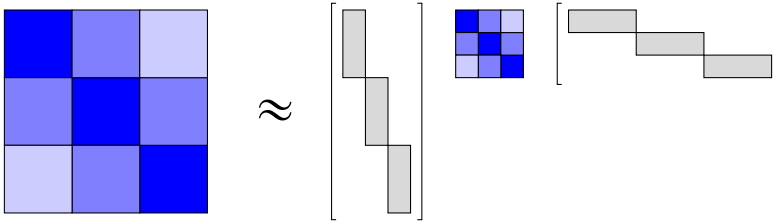
\begin{figure}[!ht]\centering
\begin{tikzpicture}[scale=0.3]
 \tikzmath{\h = 3; \w = 1;}
\tikzmath{\l = 0; \t = 3*\h; }
\foreach \x in {0,\h,2*\h}{
\filldraw[fill=blue, draw=black] (\l+ \x, \t - \x -\h ) rectangle ++(\h,\h);
}

\foreach \x in {\h,2*\h}{
\filldraw[fill=blue!50, draw=black] (\l+ \x, \t - \x) rectangle ++(\h,\h);
\filldraw[fill=blue!50, draw=black] (\l+ \x-\h , \t - \x -\h) rectangle ++(\h,\h);
}

\filldraw[fill=blue!20, draw=black] (\l+2*\h, \t -1*\h) rectangle ++(\h,\h);
\filldraw[fill=blue!20, draw=black] (\l, \t -3*\h) rectangle ++(\h,\h);
\draw (\l + 4*\h, 1.5*\h) node[fill=none,scale=2]   {$\approx$};  

 \tikzmath{\l = 5*\h; \t = 3*\h;}
\draw (\l-0.3*\w,-0.3*\w) -- (\l-0.5*\w,-0.3*\w) --  (\l-0.5*\w,\t+0.3*\w) -- (\l-0.3*\w,\t+0.3*\w);
\filldraw[fill=gray!30!white, draw=black] (\l,2*\h) rectangle ++(\w,\h);
\filldraw[fill=gray!30!white, draw=black] (\l+\w,\h) rectangle ++(\w,\h);
\filldraw[fill=gray!30!white, draw=black] (\l+2*\w,0) rectangle ++(\w,\h);
\draw (\l +3.3*\w,-0.3*\w) -- (\l+3.5*\w,-0.3*\w) --  (\l+3.5*\w,\t+0.3*\w) -- (\l +3.3*\w,\t+0.3*\w);

 \tikzmath{\l = \l +  4*\w; \t = 3*\h;}
\foreach \x in {1,2,3}{
\filldraw[fill=blue, draw=black] (\l+ \x, \t - \x) rectangle ++(\w,\w);
}

\foreach \x in {1,2}{
\filldraw[fill=blue!50, draw=black] (\l+ \x + \w, \t - \x) rectangle ++(\w,\w);
\filldraw[fill=blue!50, draw=black] (\l+ \x , \t - \x -1*\w) rectangle ++(\w,\w);
}

\filldraw[fill=blue!20, draw=black] (\l+3*\w, \t -1*\w) rectangle ++(\w,\w);
\filldraw[fill=blue!20, draw=black] (\l+1*\w, \t -3*\w) rectangle ++(\w,\w);

 \tikzmath{\l = \l + 6*\w; \t = 3*\h;}
\draw (\l-0.3*\w,\t-3.3*\w) -- (\l-0.5*\w,\t-3.3*\w) --  (\l-0.5*\w,\t+0.3*\w) -- (\l-0.3*\w,\t+0.3*\w);
\filldraw[fill=gray!30!white, draw=black] (\l,\t-1*\w) rectangle ++(\h,1);
\filldraw[fill=gray!30!white, draw=black] (\l+\h,\t-2*\w) rectangle ++(\h,1);
\filldraw[fill=gray!30!white, draw=black] (\l+2*\h,\t-3*\w) rectangle ++(\h,1);
\draw (\l + 9.3*\w,\t-3.3*\w) -- (\l+9.5*\w,\t-3.3*\w) --  (\l+9.5*\w,\t+0.3*\w) -- (\l +9.3*\w,\t+0.3*\w);
\end{tikzpicture}
\caption{Visualization of the block-structured matrix approximation.}
\label{fig:blockstructure}
\end{figure}

\paragraph{CP decomposition} First, we show how to express the tensor approximation computed in Kruskal format as a block-structured factorization.  We assume, once again that we have the CP decomposition $\TTT =  \llbracket \M{X}, \M{Y},\M{Z} \rrbracket$. Here, we are assuming that the number of terms in the CP approximation, $r$, is small relative to the problem dimension, in which case the following theorem can have important practical consequences. 
\begin{theorem}
Consider the notation and assumptions in \cref{th:tildeA}. Furthermore, let $\M{X} = \M{Q}_X \M{R}_X$ and $\M{Z} = \M{Q}_Z\M{R}_Z$ be the respective thin QR factorizations. Then, 
\begin{equation} \label{eq:blockCP} \mc{M}_{\mc{E}}[{\TTT}] = ( \MI_\ell \otimes \M{Q}_X) \struct \left( \M{F}_k  \right)_{k=1}^p (\MI_q \otimes \M{Q}_Z^\top), \end{equation}
where $\M{F}_k = \M{R}_X\diag(\MY_{k,:}) \M{R}_Z^\top \in \mb{R}^{r\times r}$ for $k=1,\dots,p$. 
\end{theorem}
\begin{proof}
Each lateral slice of $\TTT$ has the form $\M{X} \diag(\M{Y}_{k,:} ) \M{Z}^\top = \M{Q}_X \M{F}_k\M{Q}_Z^\top$. The rest of the proof follows the proof of~\cref{th:tildeA}.
\end{proof}

In this representation, the outer terms have orthonormal columns of sizes $(\ell m)\times (\ell r)$ and $(rq)\times (nq)$, whereas the middle term has block rank structure of size $(\ell r)\times (qr)$.  See \cref{fig:blockstructure} for an illustration, assuming that $r$ is relatively small.

\paragraph{Tucker format} Suppose we have the tensor approximation in the Tucker format $\TTT = [\T{G};\M{U},\M{V},\M{W}]$. The block-structured factorization is evident from observing
that the expression in \cref{eq:structTucker} can be equivalently expressed
\begin{equation} \label{eq:blockTuck} \cM_{\mc{E}}[\TTT] =  (\MI_\ell \otimes \MU) \struct \left(  \sum_{j=1}^{\ell} v_{kj} \sq(\T{G}_{:,j,:})  \right)_{k=1}^p  (\MI_q \otimes \MW^\top).\end{equation}
In this representation, the outer terms have orthonormal columns  and have sizes $(\ell m)\times (\ell r_1)$ and $(r_3q)\times (nq)$, whereas the middle term has block rank structure of size $(\ell r_1)\times (qr_3)$.

Supposing the block structured matrix $\M{A}$ is symmetric positive definite (SPSD) or symmetric positive semidefinite (SPSD), we want the approximation obtained using the tensor compression to also be SPD or SPSD, respectively. We now show how to do this in the context of block-structured factorizations.
\subsection{Preserving positive semidefiniteness}\label{ssec:spsd}

Let $\M{A} \in \mb{R}^{n\ell \times n\ell}$ be SPSD and let $\mc{A}$, $\mc{W}$, and $\cTT_{\mc{E}}[\M{A}] = \TX \in \mb{R}^{n\times \ell \times n}$ be defined as before. Note that since $\M{A}$ is symmetric we have $\ell = q$ and $m=n$, and its diagonal blocks are symmetric. Furthermore, each matrix in $\mc{A}$ (i.e. the subblocks of $\M{A}$) is either symmetric or its transpose is also contained in $\mc{A}$. For example, consider the block Toeplitz matrix Example 3 in \cref{sec:mattotens}; if the block Toeplitz matrix is symmetric, then this implies $\M{A}_1$ is symmetric and $\M{A}_{\ell+k-1} = \M{A}_k^\top$ for $k=2,\dots,\ell$. This simple observation has the following important consequence: the mode-$1$ and mode-$3$ unfoldings of $\cTT_{\mc{E}}[\M{A}]$ have the same range (see \cref{eq:mode3unfold} and the discussion preceding it), and therefore, we can use the same factor matrix to compress across each of these modes. That is, we compute a two-sided compression of the form
\[ \cTT_{\mc{E}}[\M{A}] \approx \TTT = \T{G} \times_1 \M{U} \times_3 \M{U}, \qquad \T{G} = \TTT \times_1 \M{U}^\top \times_3 \M{U}^\top, \]
where $\M{U} \in \mb{R}^{n\times r}$ has orthonormal columns. We can compute $\M{U}$, for example, from the dominant $r$ left singular vectors of either the mode-$1$ or mode-$3$ unfolding of $\TT_{\mc{E}}[\MA]$. The corresponding block-structured approximation to $\M{A}$ is 
\[ \M{A} \approx \mc{M}_{\mc{E}} [\TTT] = (\MI\otimes \M{U}) \struct\left( \M{U}^\top\M{A}_k\M{U}\right)_{k=1}^p(\MI\otimes \M{U}^\top). \]
A straightforward rearrangement shows that 
\[\mc{M}_{\mc{E}} [\TTT] = (\M{I}\otimes \M{P}) \M{A} (\M{I}\otimes \M{P})\] 
where $\M{P} = \M{U}\M{U}^\top$ is an orthogonal projector, since $\M{U}$ has orthonormal columns. Since $\M{A}$ is SPSD, it is easy to verify that its approximation $\mc{M}_{\mc{E}} [\TTT]$ is also SPSD. Note that if $\M{A}$ is symmetric (but not necessarily SPSD), we can use a Tucker compression of the form 
\[ \TTT = \T{G} \times_1 \M{U}  \times_2 \M{V} \times_3 \M{U}, \qquad \T{G} = \TTT \times_1  \M{U}^\top\times_2 \M{V}^\top \times_3 \M{U}^\top, \]
where $\M{V}$ has orthonormal columns. The resulting approximation $\mc{M}_{\mc{E}}[\TT]$ is symmetric but does preserve definiteness (if present in the original block matrix).

\subsection{Preserving positive definiteness}\label{ssec:spd} Suppose $\MT \in \mb{R}^{n\ell \times n\ell}$ is symmetric and positive definite (SPD), and we want our block-structured approximation to also be SPD. Our approach is inspired by Method 2 in~\cite{xing2018preserving}, which treats the diagonal blocks in a special way. Assume we can decompose $\MT$ as 
 \[ \MT = \MI_{\ell} \otimes \MT_0 + \sum_{k=1}^{p} \M{E}_k \otimes \MT_k,\]
 where $\MT_0$ is SPD. Such a decomposition is guaranteed since at least one diagonal block of $\MT$ is SPD. Let $\MT_0 = \ML\ML^\top$ be its Cholesky factorization. By factoring out $\MT_0$ from the  diagonal blocks, we get 
 \[ \MT = (\MI\otimes \ML) \left(\MI + \MA \right)(\MI\otimes \ML^\top), \qquad \MA = \sum_{k=1}^{p} \M{E}_k \otimes {\MA}_k ,  \]
 where $\MA_k = \ML^{-1}\MT_k\ML^{-\top} $  for $k=1,\dots,p$. Since $\M{T}$ is SPD, so is $\M{I} + \M{A}$.

 Assume that we define $\mc{A}$, $\mc{E}$, and $\cTT_{\mc{E}}[\M{A}] = \TX \in \mb{R}^{n\times \ell \times n}$ as before. Once again, we consider the tensor approximation along modes 1 and 3
 \[ \widehat{\cTT}_{\mc{E}}[\M{A}] = \TTT =  \T{G} \times_1 \M{U} \times_3 \M{U}, \qquad \T{G} = \TT \times_1 \M{U}^\top \times_3 \M{U}^\top. \]
 Therefore, this gives the matrix approximation
 \[ \mc{M}_{\mc{E}}[\TT] = (\MI\otimes \MU) \struct\left(\MU^\top\MA_k\MU\right)_{k=1}^p (\MI\otimes \MU^\top).\]
 Finally, we can then approximate $\M{T}$ as 
\[ \M{T}  \approx \widehat{\MT} :=  (\MI\otimes \ML) \left(\MI + \mc{M}_{\mc{E}}[\TT] \right)(\MI\otimes \ML^\top). \]
 In the next theorem, we show that the approximation as constructed above is also SPD. 
\begin{theorem}
Let $\MT \in \mb{R}^{n\ell \times n\ell}$ be SPD. The approximation $\widehat{\MT} \approx \MT$ constructed, as before, is SPD.
\end{theorem}
\begin{proof}
It suffices to show that $\MI + \mc{M}_{\mc{E}}[\TT]$ is SPD. Define the orthogonal projectors $\M{\Pi} = \MI \otimes \M{P}$ where $\M{P} = \M{UU}^\top$, and  $\M{\Pi}_\perp = \MI \otimes \M{P}_\perp$ where $\M{P}_\perp = \MI-\M{P} = \MI - \M{UU}^\top$. With this notation, we can write 
\[  \M{I} + \mc{M}_{\mc{E}}[\TT] = \M{I} + \M{\Pi} \MA \M{\Pi} = \M{\Pi}_\perp +  \M{\Pi} (\M{I} + \MA) \M{\Pi}. \]
Consider a nonzero vector $\V{x}\in\mb{R}^{n\ell}$ and decompose it uniquely as  $\V{x} = \V{x}_1 + \V{x}_2$ where $\V{x}_1 \in \mc{R}(\M{\Pi})$ and $\V{x}_2 \in \mc{R}(\M{\Pi}_\perp)$. The quadratic form simplifies as 
\[ \V{x}^\top(\M{I} + \mc{M}_{\mc{E}}[\TT]) \V{x}= \V{x}_2^\top\V{x}_2 + \V{x}_1^\top(\MI + \MA)\V{x}_1.\]
This is clearly nonnegative; furthermore, since $(\MI + \MA)$ is SPD, at least one of the two terms in the summation is positive. Therefore, $\M{I} + \mc{M}_{\mc{E}}[\TT]$ is SPD.
\end{proof}

It is interesting to note that we can rearrange the expression for $\widehat{\MT}$ to obtain 
\[\widehat{\MT} = \MI \otimes \MT_0 + (\MI \otimes \ML\M{P}\ML^{-1}) \struct\left(\MT_k\right)_{k=1}^p (\MI \otimes \ML^\top\M{P}\ML^{-\top}),\]
where $\ML\MP\ML^{-1} = \ML\M{UU}^\top\ML^{-1}$ is an orthogonal projector with the respect to the $\langle \V{x}, \V{y} \rangle_{\M{T}_0^{-1}} = \V{x}^\top\M{T}_0^{-1}\V{y}$ inner-product.

\section{Applications and Numerical Experiments} \label{sec:apps}
In this section, we demonstrate the utility of our formulation and the performance of the various algorithms proposed in the previous sections. The first application (\cref{ssec:ssid}) we consider is accelerating the computation of the eigensystem realization algorithm, which involves large block-Hankel matrices. The second application (\cref{ssec:spacetime}) arises from space-time covariance matrices and involve block-Toeplitz matrices. In \cref{ssec:blocktrid} we obtain Kronecker sum approximations to several sparse test matrices obtained from the SuiteSparse matrix collection. 

\subsection{Subspace System identification}\label{ssec:ssid}
Consider the linear time-invariant discrete dynamical system 
\[ \begin{aligned}\V{x}_{k+1} = & \> \M{A}\V{x}_k + \M{B}\V{u}_k \\ \V{y}_{k+1} = &\> \M{C}\V{x}_{k+1} + \M{D}\V{u}_{k+1} . \end{aligned}\]
Here $\M{A} \in \mb{R}^{d\times d}$ is the state transition matrix, $\M{B}\in \mb{R}^{d\times n}$ is the input matrix, $\M{C} \in \mb{R}^{m\times d}$ is the output matrix and $\M{D} \in \mb{R}^{m\times n}$ is the feedthrough matrix 
The goal in system identification is to recover the system matrices $(\M{A},\M{B},\M{C},\M{D})$ (up to a certain similarity transformation) governing the dynamics of the system from the inputs $\{\V{u}_k\}$ and outputs $\{\V{y}_k\}$. When the inputs are of the impulse type, a special algorithm known as the Eigensystem Realization Algorithm (ERA) can be used for system identification~\cite{kung1978new,juang1985eigensystem}. 

In this setting, we are given, as data, the {\em Markov parameters}
\begin{equation}
    \M{h}_{k} = \left\{\begin{array}{ll}
        \M{D} & k = 0  \\
        \M{C}\M{A}^{k-1}\M{B} & k > 0. 
    \end{array} \right. 
\end{equation}
Recovering the system matrix $\M{D}$ is straightforward since it is the first Markov parameter. In what follows, we focus on recovering $\M{A},\M{B},$ and $\M{C}$. 
Then, a block-Hankel matrix is formed with the Markov parameters such that there are $s$ block rows and columns of size $m \times n$ each:
\[ \MH_s = \bea \MCC \MB & \MCC \MA \MB & \cdots & \MCC \MA^{s-1} \MB \\
                     \MCC \MA \MB & \MCC \MA^2 \MB & \cdots & \MCC \MA^{s} \MB \\
                      \vdots & \vdots & \ddots & \vdots \\
                     \MCC \MA^{s-1} \MB & \MCC \MA^{s} \MB & \cdots & \MCC \MA^{2s-2} \MB \eea  \in \mb{R}^{(ms)\times (ns)}.\]
The value of $s$ might also be large in practice, with $s \approx \mc{O}(10^3 - 10^5)$. A review of the computational challenges involved in storing and factorizing $\MH_s$ are reviewed in~\cite{minster2020efficient}.

 We offer a brief derivation of ERA here. Assuming that the system is observable and controllable (see~\cite[Lemma 3.5]{verhaegen2007filtering} for more details), we have $\text{rank}(\MH_s)=d$ and  we can factorize $\MH_s$ as
\[ \MH_s = \underbrace{\bea \MCC \\ \MCC \MA \\ \vdots \\ \MCC \MA^{s-1} \eea}_{\M{O}_s} \underbrace{[ \MB, \MA \MB, \cdots, \MA^{s-1} \MB ]}_{\M{C}_s},\]
where $\M{O}_s$ is the observability matrix and $\M{C}_s$ is the controllability matrix, which both have rank $n$. The trick is to partition the observability matrix $\M{O}_s$
\[ \M{O}_s^{(f)} {\MA} = \M{O}_s^{(b)} \]
where $\M{O}_s^{(f)}$ is the leading $ms-m$ rows of $\M{O}_{s}$ and $\M{O}_s^{(b)}$ contains rows $m+1$ to $ms$ of $\M{O}_s$. We can recover the matrix $\MA$ (up to a similarity transformation) by solving this problem in the least squares sense. That is, we write 
\[ \tilde{\MA} = [\M{O}_s^{(f)}]^\dagger \M{O}_s^{(b)}.\]
In practice, the observability and controllability matrices are not explicitly available, but we can instead work with a factorization of the form $\MH_s = \M\Theta_s \M\Gamma_s^\top$, where $\M\Theta_s$ has the same dimensions as $\M{O}_s$. For example, we can construct such a factorization using the reduced SVD. Furthermore, if we desire model reduction in addition to system identification, we work with the SVD of $\MH_s$ truncated to target rank $r < d$.

Once the state transition matrix $\MA$ is recovered, we can also recover the matrices $\M{B}$ and $\M{C}$. The details are available in other sources (e.g.,~\cite{verhaegen2007filtering,minster2020efficient}). The step involving the computation of the SVD of $\MH_s$ dominates the overall computational cost, since it is cubic in the number of time points $s$. Other methods to mitigate this computational cost are available in~\cite{minster2020efficient,kramer2016tangential}, etc. In this section, we exploit the tensor-based computational framework for computing structured matrix approximations developed in \cref{sec:kron}.

\paragraph{Matrix-to-tensor mapping} We now show how to map the structured block matrix $\MH_s$ into tensor form for compression using the tensor techniques. Let $\eta_k$ be defined as 
\[\eta_k = \left\{ \begin{array}{ll} {k} & 1 \leq k  \leq s \\ {2s-k} & s  < k < 2s-1. \end{array}\right. \]
Then each $\M{E}_k$ is a Hankel matrix with $0$'s everywhere except for indices where $i+j-1 = k$, where it takes the value $\sqrt{\eta_k}$ and
\[ \mc{E} = ( \M{E}_1,\ldots,\M{E}_{2s-1} ) \qquad \mc{A} = (\M{h}_1,\dots,\M{h}_{2s-1} ).  \]
Therefore, the tensor $\cTT_{\mc{E}}[\MH_s]_{:k:} = \sqrt{\eta_k}\twist(\M{h}_k)$ for $k=1,\dots,2s-1$.

We compute a Tucker decomposition of multi rank $(r_1,r_2,r_3)$ of the form 
\[ \cTT_{\mc{E}}[\MH_s] \approx  \TTT = \widehat{\cTT}_{\mc{E}}[\MH_s]  = \T{G} \times_1 \MU \times_2 \M{V} \times_3 \M{W}, \]
where $\T{G} = \TTT \times_1  \M{U}^\top \times_2\M{V}^\top \times_3\M{W}^\top$. This low-rank decomposition is computed using the HOSVD algorithm. In this example problem the number of outputs and inputs are $m=155$, $n = 50$ respectively, and we take $s = 100$. The size of the tensor $\cTT_{\mc{W}}[\MH_s]$ is, therefore, $155 \times 199\times 50.$ In Figure~\ref{fig:tuckerera}, we plot the singular value decay in each mode-unfolding. We observe that the singular values decay in all three mode-unfoldings, with a very sharp drop in mode-2.  We can then write the approximate block-structured matrix as 
\[ \widehat{\MH}_s = (\M{I}\otimes \M{U}) \struct\left\{ \sum_{j=1}^{r_2} v_{kj} \sq(\T{G}_{:,j,:}) \right\} _{k=1}^{2s-1} (\M{I}\otimes \M{W}^\top).  \]

\begin{figure}[!ht]
    \centering
    \includegraphics[scale=0.4]{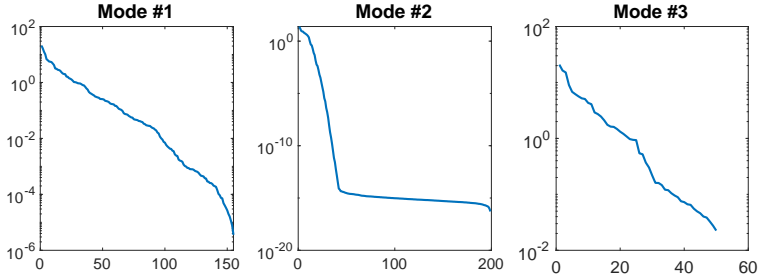}
    \caption{This plot shows the singular value decay of each mode unfolding of $\cTT[\MH_s]$, corresponding to the power system application problem.}
    \label{fig:tuckerera}
\end{figure}

To perform system identification, it is sufficient to work with the SVD of 
\[\widehat{\MH}_s^\text{Tucker} = \struct\left( \sum_{j=1}^{r_2} v_{kj} \sq (\TG_{:,j,:}) \right)_{k=1}^{2s-1},\] 
since the SVD of $\widehat{\MH}_s$ follows in a straightforward manner. Since $\widehat{\MH}_s^\text{Tucker}$ is still in the block-Hankel format, we can take advantage of its structure to efficiently compute its dominant singular values. In particular, we apply the Randomized ERA algorithm~\cite[Section 3.1]{minster2020efficient} to $\widehat{\MH}_s^\text{Tucker}$. We call this the `TuckerRandERA' approach. This has a computational cost $\mc{O}(r_1r_3s\log_2 s)$ compared to $\mc{O}(mn\min\{m,n\}s^3)$ using the naive approach. Similarly, the storage cost of the Tucker approach is $\mc{O}(r_1r_3s)$ but for the naive approach this is $\mc{O}(mns^2)$ (since the entire matrix $\MH_s$ needs to be stored). 

\begin{figure}[!ht]
    \centering
    \includegraphics[scale=0.4]{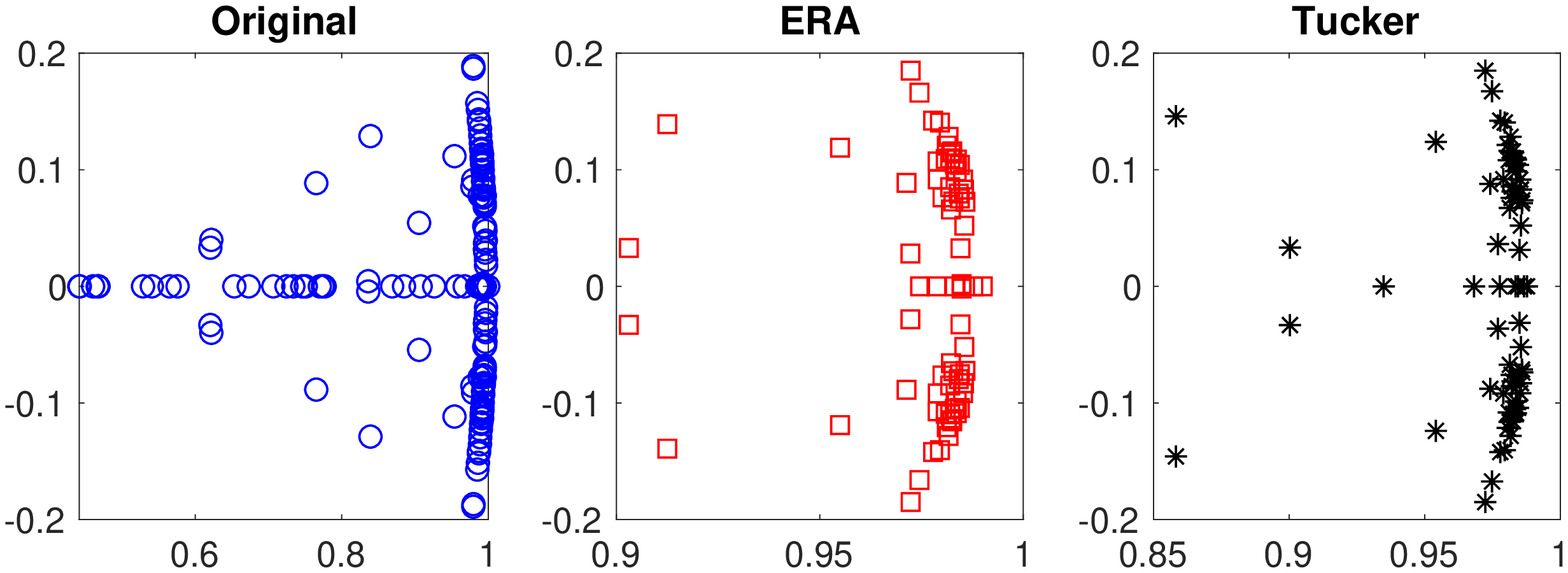}
    \caption{(Left) displays the eigenvalues of the original system, (middle) displays the eigenvalues using ERA, and (right) displays the realized eigenvalues computed using our TuckerRandERA approach.}
    \label{fig:tuckereigs}
\end{figure}

In the following computational experiment,  we pick the target rank $(56,30,30)$ based on the decay of the singular values of the mode unfolding (see Figure~\ref{fig:tuckerera}). This is a significant compression considering the size of the original tensor is $155\times 199\times 50$. Based on~\cite{minster2020efficient}  the target rank for the ERA was chosen to be $r=75$. In Figure~\ref{fig:tuckereigs}, we plot the eigenvalues of $\M{A}$. The left panel displays the eigenvalues of the original system, the middle panel displays the eigenvalues using ERA, and the right panel displays the realized eigenvalues computed using the TuckerRandERA approach. As is seen from the plot, the eigenvalues computed both using the ERA and the TuckerRandERA approach are in good agreement, from a visual perspective. The Hausdorff distance between the eigenvalues identified using ERA and TuckerRandERA is approximately $5.47\times 10^{-2}$ (see Section 5 in~\cite{minster2020efficient} for an explanation of this measure of accuracy). The accuracy improves significantly if a larger value of $s$ is used. The run time for ERA was approximately $68$ seconds, and the time for the TuckerRandERA approach is approximately $0.79$ seconds (including precomputation of the Tucker decomposition and averaged over three runs). Both the accuracy and the computational benefits improve significantly if a larger $s$ is used.

\paragraph{Relation to TERA} A related approach proposed in the literature is Tangential Interpolation ERA (TERA)~\cite{kramer2016tangential}. In our notation, the TERA approximation to $\MH_s$ is 
\[ \MH_s^\text{TERA} = (\MI \otimes \MU) \struct\left( \widetilde{\M{h}}_k\right)_{k=1}^{2s-1}  (\MI \otimes \MW^\top), \]
where $\widetilde{\M{h}}_k = \MU^\top\M{h}_k\MW = \MU^\top\MCC\MA^{k-1}\MB\MW$, are the projected Markov parameters. We construct the tensor $\cTT[\MH_s]_{:k:} = \twist(\M{h}_k)$ with the Markov parameters as the lateral slices. Note that this tensor is constructed in a slightly different way than in Section~\ref{sec:mattotens}, where the lateral slices are weighted by the frequency with which they appear in the block matrix. It is relatively straightforward to show that this corresponds to a Tucker-2 approximation of $\cTT[\MH_s]$ of the form 
\[ \cTT[\MH_s] \approx \widehat{\cTT}[\MH_s]  = \T{G} \times_1 \MU \times_2 \M{I} \times_3 \M{W}, \]
where $\T{G} = \cTT[\MH_s]\times_1  \M{U}^\top  \times_3\M{W}^\top$.

The following facts are worth pointing out regarding this connection to TERA. First, the structure of TERA is similar to that obtained using the Tucker approach described in this section if $\M{V} = \M{I}$; that is, we compress along modes 1 and 3 only. Second, the computational cost of TERA and using the Tucker approach is comparable; since both algorithms preserve the block-Hankel structure, they can be accelerated using the Randomized ERA approach. Third, using the Tucker decomposition across all three modes enables us to use additional structure that the TERA that is, presumably, neglecting. Finally, the TERA approximation treats all the Markov parameters ``equally'' and does not consider the frequency of occurrence in the block matrix $\MH_s$.

\subsection{Space-time covariance matrices}\label{ssec:spacetime}
In geostatistical applications, it is important to represent the unknown quantities that vary over space and time by random processes~\cite{gneiting2006geostatistical}. Gaussian random processes are often used in this context and are completely characterized by the mean and covariance functions. Consider the Gaussian random field $Z(\V{x},t)$; assuming that it has a zero mean and finite second moment, we can represent the covariance function between two points in space-time as
\[ C((\V{x}_1,t_1), (\V{x}_2,t_2)) = \text{Cov}(Z(\V{x}_1,t_1),Z(\V{x}_2,t_2)). \]
Depending on the the type of covariance function used and the number of points in space and time (e.g., nonseparable covariance function), the resulting space-time covariance matrices can be large, dense, and infeasible both from a storage and computational perspective.  An accurate and efficient approximation of the resulting space-time covariance matrix is highly desirable in many applications. We consider the class of covariance matrices that are weakly stationary and isotropic, so that 
\[ C((\V{x}_1,t_1), (\V{x}_2,t_2)) = \varphi(\|\V{x}_1-\V{x}_2\|_2, |t_1-t_2|).\]

Let $\{\V{x}_j\}_{j=1}^N$ be a set of points in space and $\{t_j\}_{j=1}^T$ be a set of points in time. The points in space are allowed to be unstructured, but the points in time are considered to be equidistant. Consider the space-time covariance matrix $\M{C} \in \mb{R}^{(NT)\times (NT)}$ which is block Toeplitz
\[ \M{C} = \bmat{ \M{C}_1 & \M{C}_2 & \dots & \M{C}_T \\ \M{C}_2 & \M{C}_1 & \dots & \M{C}_{T-1} \\  \vdots & \ddots & \ddots & \vdots \\ \M{C}_{T} & \dots & \M{C}_{2} &  \M{C}_1}\in \mb{R}^{(NT)\times (NT)}, \]
where the blocks $\M{C}_i \in \mb{R}^{N\times N}$ have the entries
\[ [\M{C}_i]_{jk} = \varphi(\|\V{x}_j-\V{x}_k\|_2, |t_1-t_i|) \qquad j,k=1,\dots,N. \]
We consider the block low-rank approximation as proposed in Section~\ref{sec:blockstructured}; however, since covariance matrices are symmetric positive (semi) definite, when constructing an approximation it is important to ensure that it preserves definiteness.

The mapping to tensors has already been described in Example~\ref{ex:btoep} of Section~\ref{sec:mattotens} and we do not repeat it here. Note that $\ell = q = T$ and $m = n = N$. To construct the low-rank approximation, we follow the approach in Section~\ref{ssec:spsd}, and compute the rank-$(r,T,r)$ Tucker decomposition of $\cTT_{\mc{E}}[\M{C}] \in \mb{R}^{N\times T\times N}$ to obtain 
\[  \cTT_{\mc{E}}[\M{C}] \approx  \TT := \T{G} \times_1 \M{U} \times_3 \M{U} .  \]
To compute this decomposition, because of the symmetry of the problem, we only need to compute the left singular vectors $\M{U} \in \mb{R}^{N\times r}$ of the mode-1 unfolding of $\cTT_{\mc{E}}[\M{C}]$ (corresponding to the $r$ dominant singular values) and then compute the core tensor $\T{G} := \cTT_{\mc{E}}[\M{C}]\times_1 \M{U}^\top \times_3 \M{U}^\top$. If the SVD is used to compute $\M{U}$, then the cost is $\mc{O}(TN^3)$, but can be lowered to $\mc{O}(rN^2T)$ using randomized SVD~\cite{halko2011finding}. This cost can be further lowered to $\mc{O}(rNT\log N)$ using the $\mc{H}$-matrix approach~\cite{litvinenko2020hlibcov}, but we do not pursue this here. 
We then obtain the block-structured approximation to $\M{C}$ as 
\[ \widehat{\M{C}} = \mc{M}_{\mc{E}} [\TTT] =  (\MI\otimes \M{U}) \struct\left(  \M{U}^\top\M{C}_k\M{U} \right)_{k=1}^T(\MI\otimes \M{U}^\top). \]
The block-structured approximation $\widehat{\M{C}}$ is efficient both in terms of storage and computation. The cost of storing $\widehat{\M{C}}$ implicitly is $\mc{O}(rN + r^2T).$  In contrast, the original matrix has the storage cost $\mc{O}(N^2T)$, if only the blocks $\M{C}_j$ are stored. The cost of a matvec with $\widehat{\M{C}}$ is $\mc{O}(rNT+ r^2T^2)$, whereas a matvec with $\M{C}$ can be computed in $\mc{O}(N^2T^2)$. Both of these costs can be lowered, if we exploit the block-Toeplitz structure; in particular, we can reduce the factor $T^2$ to $T\log T$. 
\begin{figure}[!ht]
    \centering
    \includegraphics[scale=0.5]{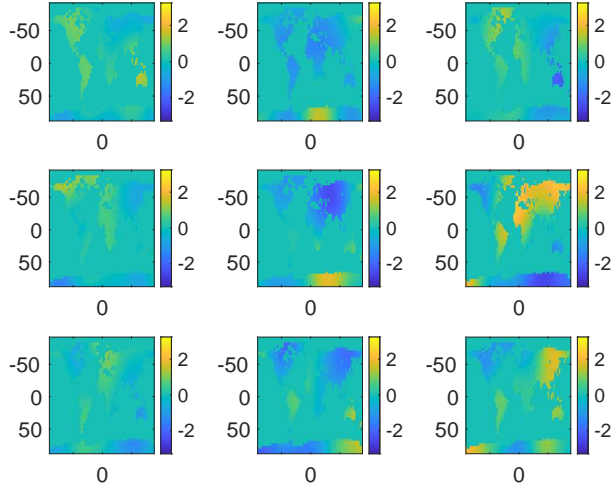}
    \caption{Samples from the distribution $\mc{N}(\V{0},\widehat{\M{C}}+\delta \M{I})$. The rows correspond to the time indices $1,5,15$ and the columns correspond to three different samples. }
    \label{fig:spacetime}
\end{figure}

In our first experiment, we check the accuracy of the block-structured approximation. We pick the covariance kernel corresponding to the function 
\[ \varphi(r,\tau) = \exp\left(-((r/90)^2 + (\tau/0.5)^2)\right).\] 
We divide the world map into a $4^\degree \times 4^\degree$ grid, the points $\V{x}_j$ correspond to the land masses. In this particular instance, we have $N=1351$ points and we choose $T=30$ evenly spaced time points in the interval $[0,1]$. To check the accuracy of the compression, we use the following relative error
\[ \text{relerr} = \frac{|\trace(\M{C}) -\trace(\widehat{\M{C}})|}{\text{trace}(\M{C})}.\]
Since $\M{C}$ cannot be formed explicitly, this metric is beneficial since it is easy to compute. Observe that $\trace(\M{C}) = NT$ and $\trace(\widehat{\M{C}}) = T\trace(\M{U}^\top\M{C}_1\M{U}).$ When $r=30$, we obtain the relative error relerr $\approx 3.35 \times 10^{-4}$, but the compression ratio
\[ \text{ratio}_\text{storage} = \frac{r^2T + Nr}{N^2T} \approx 0.002. \]
In other words, at only a fraction of the storage costs, we obtain an approximation to $\M{C}$ that is reasonably accurate. In Figure~\ref{fig:spacetime}, we plot samples from $\mc{N}(\V{0},\widehat{\M{C}}+\delta \M{I})$, where $\delta = 10^{-8}$ is known as the nugget parameter added to ensure positive definiteness. The samples were computed using the Lanczos approach~\cite{chow2014preconditioned}. 

\subsection{Test Matrices from SuiteSparse Collection}\label{ssec:blocktrid} In this example, we consider a set of sparse block-tridigonal matrices from  the SuiteSparse collection of matrices~\cite{kolodziej2019suitesparse}. The tuples $\mc{A}, \mc{E}$ are described in Example 2 of \Cref{sec:mattotens}, so we will not repeat the matrix-to-tensor mappings. Note that $\ell  = q$,  $m=n$, and $p = 3\ell - 2$. For each matrix $\M{A}$, we produce a Kronecker sum approximation by first computing a compressed representation of the form
\[ \T{T} =  \T{G} \times_2 \M{V} \qquad \T{G} = \T{X} \times_2\M{V}^\top,\]
where $\T{X} = \cTT_{\mc{E}}[\M{A}]$ and $\M{V} \in \mb{R}^{p \times r}$ so that $\TG \in \mathbb{R}^{n \times r \times n}$. 
Here $r$ is the target rank of the compression. We chose to compress along mode-2 since for each example, the singular values of the mode-2 unfolding decayed rapidly. The matrix $\M{V}$ is computed using the first $r$ left singular vectors of $\T{X}_{(2)}$. 

\begin{table}[!ht]
    \centering
    \begin{tabular}{l|c|c|c|c|c}
        Name & $\ell$ &  $n$ & Target rank $r$ & Relative Error & Compression   \\ \hline
        \verb|pde2961| & $63$ & $47$ & $20$ & $8.36\times 10^{-10}$ & $0.4470$ \\
        \verb|t2d_q4| & $99$ & $99$ & $5$ & $2.93\times 10^{-15}$ & $0.034$\\
        \verb|t2d_q9| & $99$ & $99$ & $5$ & $2.93\times 10^{-15}$& $0.034$\\ 
        \verb|fv2| & $99$ & $99$ & $5$ & $2.28\times 10^{-15}$ & $0.034$\\ 
        \verb|chem_master1| & $201$ & $201$ & $5$ & $1.79\times 10^{-15}$ & $0.030$\\
        \verb|ecology1| (*) &$500$ & $1000$ & $5$ & $6.10\times 10^{-15}$ & $ 0.009$ 
    \end{tabular}
    \caption{For each test problem, we report the name of the matrix, the number of block rows $\ell$, the size of each block $n$, the target rank used, the relative error, and the compression ratio. For the \texttt{ecology1} test case, we used the leading principal submatrix of size $500000\times 500000$.}
    \label{tab:kronsum}
\end{table}

The compressed representation is then converted to a Kronecker sum approximation of the form $\M{A} \approx \Mhat{A} = \sum_{j=1}^r \M{C}_j \otimes \M{D}_j$, where
\[   \M{C}_j = \struct (v_{kj} )_{k=1}^p \in \mb{R}^{\ell \times \ell}, \quad \M{D}_j =  \sq(\T{G}_{:,j,:}) \in \mb{R}^{n\times n}.\]
Note here that $\M{C}_j$ is a tridiagonal matrix, so the resulting approximation $\Mhat{A}$ is also block tridiagonal. Furthermore, note that for the chosen test problems, since the matrices $\M{A}_j$ for $j=1,\dots,p$ are tridiagonal, the matrices $\M{D}_j$ for $j=1,\dots,r$ are also tridiagonal matrices as such structure will be preserved in the lateral slices of $\TG$. This makes the resulting Kronecker sum approximation efficient to store, providing $r \leq p$.

In \cref{tab:kronsum}, we list the name of the matrix, the relevant problem dimensions, the target rank, and the resulting relative error in the Frobenius norm $\|\M{A}-\Mhat{A}\|_F/\|\M{A}\|_F$. We also report the compression ration
\[ \text{ratio}_\text{storage} = \frac{pr + \text{nnz}(\T{G})}{\text{nnz}(\MA)}.\]
For the test problem \verb|ecology1|, we made two modifications to reduce the runtime of the problem: first we used the leading principal submatrix of size $500000\times 500000$; second, to compute the matrix $\M{V}$, we use a randomized approach~\cite{minster2020randomized}. That is, we compute the $\M{V}$ to be the left singular vectors of the mode-2 unfolding of the tensor $\M{Y} = \T{X} \times_1 \M{\Omega} \times_3 \M\Psi$. Here $\M\Omega,\Psi \in \mb{R}^{n\times r}$ are chosen to be independent standard Gaussian random matrices. In \cref{tab:kronsum}, it is seen that in each we obtain a Kronecker sum approximation that is accurate, efficient to compute with and to store.

\section{Multilevel block-structure}\label{sec:multilevel}
The discovery of latent tensor structure in the block matrices can be extended to multiple levels of block structure. As a motivating example, consider a Block Toeplitz matrix with Toeplitz blocks (BTTB matrix). Under the current formulation, we would represent this as a three-dimensional tensor with Toeplitz matrices as lateral slices, but ignore the Toeplitz structure within each lateral slice. We now develop a formulation that is capable of handling structure on every level on which it exists.  
\subsection{Multilevel structured approximation}
Let $\M{A} \in \mb{R}^{M\times N}$ with $L$ levels of block structure, where $M = m \left(\prod_{j=1}^L\ell_j\right)$ and $N = n \left(\prod_{k=1}^L q_j\right)$.  That is, at level $1$, the matrix $\M{A}$ has $\ell_1 \times q_1 $ blocks of size $\left(M/\ell_1 \right)\times \left(N/q_1\right)$; at level $L$, the matrices are of size $m\times n$.  Note that in our model, the type of block structure (e.g., banded, Toeplitz, Hankel, etc.) at each level is fixed.  However, the type of block structure can vary across levels.   

\begin{figure}[!ht]
\centering
\begin{minipage}[c]{1.75in}

\begin{tikzpicture}[scale=0.3]
 \tikzmath{\h = 3; \w = 1;}
\tikzmath{\l = 0; \t = 3*\h; }
\foreach \x in {0,\h,3*\h}{
\filldraw[fill=blue, draw=black] (\l+ \x, \t - \x -\h ) rectangle ++(\h,\h);
}
\foreach \x in {\h,2*\h}{
\filldraw[fill=blue!50, draw=black] (\l+ \x, \t - \x) rectangle ++(\h,\h);
\filldraw[fill=blue!50, draw=black] (\l+ \x-\h , \t - \x -\h) rectangle ++(\h,\h);
}

\filldraw[fill=blue!20, draw=black] (\l+2*\h, \t -1*\h) rectangle ++(\h,\h);
\filldraw[fill=blue!20, draw=black] (\l, \t -3*\h) rectangle ++(\h,\h);

\filldraw[fill=blue, draw=black](2*\h, 0)
rectangle ++(\h,\h); 

\filldraw[fill=blue!50, draw=black](2*\h,-1*\h)
rectangle ++(\h,\h);

\filldraw[fill=blue!50, draw=black](3*\h,0)
rectangle ++(\h,\h);

\filldraw[fill=blue!20, draw=black](1*\h,-1*\h)
rectangle ++(\h,\h);
\filldraw[fill=blue!20, draw=black](3*\h,1*\h)
rectangle ++(\h,\h);
\draw (3*\h,2*\h)
rectangle ++(\h,\h);
\draw (0,-1*\h)
rectangle ++(\h,\h); 
\end{tikzpicture}
\end{minipage} 
\begin{minipage}[c]{4in}
 	 \[ \MA_i = \begin{bmatrix} \MA^{(i,1)} & 0 & 0 & 0 & 0 \\
		0 & \MA^{(i,2)} & 0 & 0 & 0\\ 0 & 0 & \MA^{(i,3)} & 0 & 0 \\ 0 & 0 & 0 & \MA^{(i,4)} & 0 \\ 0 & 0 & 0 & 0 & \MA^{(i,5)}   \end{bmatrix} \qquad i = 1,2,3 \]
		\end{minipage}
\caption{Example where $L=2$, $l_1=q_1 = 4$, $l_2 = q_2 = 5$; The level 1 structure is block-symmetric Toeplitz (top row), the level 2 structure is block diagonal (bottom row). Each colored square on the left denotes a block diagonal matrix $\MA_i$, where  $i=1,2,3$ represents the non-redundant information on the first level; each subblock $\MA^{(i,j)}$, $j=1,\ldots,5$ is $m \times n$.} 
\label{fig:exmultilevel}
\end{figure}

\begin{example} Consider a $4 \times 4$ block-pentadiagonal Toeplitz matrix block symmetry and each block is itself a $5 \times 5$ block-diagonal matrix such that each block on the block diagonal is an $m \times n$ matrix with no additional structure to exploit (see \cref{fig:exmultilevel}).  In this example, $L=2$, $\ell_1 = q_1 = 4$, $\ell_2=q_2=5$.  The total dimension of the matrix is $m( 4 \times 5) \times n (4 \times 5)$, but because of the outer block Toeplitz structure and symmetry, the matrix can be represented by storing a total of $15=3 \times 5$, $m \times n$ matrices. However, we can keep those $15$ $m \times n$ matrices in a 4th order, $m \times 3 \times 5 \times n$ tensor.          
\end{example}

\subsection{Structured matrix to tensor mapping} In general, we can then write 
\[ \M{A} = \sum_{i_1=1}^{p_1} \cdots \sum_{i_L=1}^{p_L} \M{E}_{i_1}^{(1)} \otimes \cdots \otimes
\M{E}_{i_L}^{(L)} \otimes 
\sqrt{\eta_{i_1}^{(1)}\cdots\eta_{i_L}^{(L)} }
\M{A}^{(i_1,\dots,i_L)} ,\]
where $\MA^{(i_1,\dots,i_L)}$ are the $m \times n$ non-redundant blocks at level $L$.  A similar setup was also used in~\cite{pestana2019preconditioners}. The matrices $\M{E}_{k}^{(j)} \in \mb{R}^{\ell_j\times q_j}$, with entries in $\{0,\frac{1}{\sqrt{\eta_k^{(j)}}} \}$, represent the different mapping matrices at level $j$, where $1 \leq j \leq L$. At each level $j$, there are $p_j$ such matrices where $1 \leq p_j \leq \ell_jq_j$. Define the tuple of mapping matrices
\[ \mc{E}^{(j)} = \left( \M{E}_{1}^{(j)},\dots, \M{E}_{p_j}^{(j)}\right) \qquad j = 1,\dots,L, \]
and $\mc{E} = \left( \mc{E}^{(1)}, \dots, \mc{E}^{(L)}\right).$   
For completeness, we note that in our running example, $p_1=3, p_2 = 5$, $\eta^{(1)}_1=4, \eta^{(1)}_2 = 3, \eta^{(1)}_3 = 2$ and $\eta^{(2)}_{i_2} = 1, i_2 = 1,\ldots,5$.

We can express the blocks as lateral slices of a tensor  $\cTT_{\mc{E}}[\M{A}] \in \mb{R}^{m\times p_1 \times \dots  \times p_L \times n}$, of order $(L+2)$, such that\footnote{We will have to choose an ordering for the indicies. This description is independent of the ordering, as long as it is consistent.  For example, one could use lexicographical, first running over $i_1$, with all others fixed, then $i_2$ etc.}   
\[ \TX_{:,i_1,\dots,i_L,:} = 
\twist(\M{A}^{(i_1,\dots,i_L)}). \]

In our running example, this means that $\TX_{i,j,k,\ell} = \MA^{(j,k)}_{i,\ell}$.  
Consistent with the third order case, we express the mode-$(L+2)$ unfolding as\footnote{Consistent with ordering first in $i_1$, then $i_2$, etc.} 
\[  \TX_{(L+2)} = [\sq(\TX_{:,1,\ldots,1,:})^\top,
\sq( \TX_{:,2,1,\ldots,1,:})^\top, \cdots, 
\sq( \TX_{:,p_1,\ldots,p_L,:})^\top ].  \]

The tensor-to-matrix mapping $\mc{M}_{\mc{E}}: \mb{R}^{m\times p_1 \times \dots \times p_L \times n} \rightarrow \mb{R}^{M\times N}$ is 
\[\mc{M}_{\mc{E}}[\TX] =   \sum_{i_1=1}^{p_1} \cdots \sum_{i_L=1}^{p_L} \M{E}_{i_1}^{(1)} \otimes \cdots \otimes \M{E}_{i_L}^{(L)} \otimes \sq(\TX_{:,i_1,\dots,i_L,:}) .\] 
It is easy to verify that $\mc{M}_{\mc{E}}[\cTT_{\mc{E}}[\M{A}]] = \M{A}.$

In the next section, we will show how this new framework encapsulates some special cases mentioned in the literature.

\subsection{Approximation error} Suppose we construct a low-rank tensor approximation $\widehat{\cTT}_{\mc{E}}[\M{A}]$. The next result generalizes \cref{lem:tenapprox} to the multilevel case and relates the error in the tensor approximation to the error in the block-matrix approximation in the Frobenius norm.
\begin{lemma} \label{lem:multiapprox} Let $\MA$, $\mc{M}_{\mc{E}}$, and $\mc{T}_{\mc{E}}$ be defined as in this section, and let $\widehat{\cTT}_{\mc{E}}[\M{A}] \approx {\cTT}_{\mc{E}}[\M{A}]$. The error in $\mc{M}_{\mc{E}}[\widehat{\cTT}_{\mc{E}}[\M{A}]]$ satisfies
\[ \| \M{A} - \mc{M}_{\mc{E}}[\widehat{\cTT}_{\mc{E}}[\M{A}]]\|_F  = \|\cTT_{\mc{E}}[\M{A}]- \widehat{\cTT}_{\mc{E}}[\M{A}]\|_F\]
\end{lemma}
\begin{proof}
The proof follows from an iterative application of the technique in the proof of \cref{lem:tenapprox}.
\end{proof}
Once again, we see that the error in the Frobenius norm in the matrix approximation is precisely equal to the error in the Frobenius norm in the tensor approximation and that the error is independent of the tensor compression technique or the format that is used. We now show how to convert the compressed representation in the Tucker format to obtain a sum of Kronecker products.

\subsection{Structured matrix approximation} As in \cref{sec:kron}, a low rank approximation to ${\cTT}_{\mc{E}}[\M{A}] $ can be computed either in the CP or Tucker format, using any appropriate low-rank approximation algorithm. Let us denote $\TX = {\cTT}_{\mc{E}}[\M{A}]$.  Suppose we now use a  Tucker approximation 
\begin{equation} \label{eq:multiTX} \TX \approx {\TTT} = \T{G} \times_1 \M{X} \left(\bigtimes_{j=2}^{L+1} \M{Y}^{(j-1)}\right) \times_{L+2} \M{Z}, \end{equation}
where the factor matrices $\M{Y}^{(j)} \in \mb{R}^{p_j\times k_j}$ where $k_{j} \le p_{j}$, the factor matrix $\MX \in \mb{R}^{m\times k_x}$ and the factor matrix $\MZ\in\mb{R}^{n\times k_z}$.

Then comparing the mode-$(L+2)$ unfoldings, 
\[ \TX_{(L+2)} \approx \MZ \T{G}_{(L+2)} (\MY^{(L)})^\top \otimes \cdots (\MY^{(1)})^\top \otimes \MX^\top.  \]
It follows that 
\[ \sq (\TX_{:,i_1,\ldots,i_L,:})^\top \approx
\MZ \sum_{j_1=1}^{k_1} \cdots 
\sum_{j_L=1}^{k_L} \MY^{(L)}_{i_L,j_L} \MY^{(L-1)}_{i_{L-1},j_{L-1}} \cdots \MY^{(1)}_{i_1,j_1} \sq( \T{G}_{:,j_1,\ldots,j_L,:} )^\top \MX^\top.  \]

This gives our matrix approximation as
\begin{equation} \label{eq:matrixTX} \mc{M}_{\mc{E}}[{\TTT}] = (\M{I} \otimes \M{X}) \left( \sum_{i_1=1}^{p_1} \cdots \sum_{i_L=1}^{p_L} \bigotimes_{j=1}^L\M{E}_{i_j}^{(j)}  \otimes \M{H}^{(i_1,\dots,i_L)} \right) (\M{I} \otimes \M{Z}^\top),\end{equation}
where 
\begin{equation} \label{eq:defineH} \MH^{(i_1,\ldots,i_L)} = {\sum_{j_1=1}^{k_1}\dots }\sum_{j_L=1}^{k_L} \MY^{(L)}_{i_L,j_L} \MY^{(L-1)}_{i_{L-1},j_{L-1}} \cdots \MY^{(1)}_{i_1,j_1} \sq( \T{G}_{:,j_1,\ldots,j_L,:})
.\end{equation}
Inserting \cref{eq:defineH} into \cref{eq:matrixTX}, and simplifying, we arrive at the approximation
\[ \mc{M}_{\mc{E}}[{\TTT}] =  
\sum_{j_1=1}^{k_1} \cdots \sum_{j_L = 1}^{k_L} 
 \bigotimes_{t=1}^L
  \mathsf{struct}_{\mc{E}^{(t)}}(\MY^{(t)}_{:,j_t}) \otimes 
   \left(\MX \sq( \T{G}_{:,j_1,\ldots,j_L,:} )\MZ^\top\right) . \]

If the number of columns, $k_x$, in $\MX$ is less than $m$ and/or the number of columns $k_z$, in $\MZ$ is less than $n$, this approximation obviously has lower rank than $\MA$. Note that the slices $\sq (\T{G}_{(:,j_1,j_2,\ldots,j_L,:)})$ are of size $k_x \times k_z$.  

\subsection{Use Case: Relation to Previous Literature}
The framework we have described here subsumes, and in some cases enhances, some special cases that have been described elsewhere in the literature.
For example, we can show that our method can be used to derive exactly the same optimal Kronecker decomposition as in
\cite{KammNagy} for BTTB matrices by taking $L = 2$, and $m=n=1$ above.

In this section, we use our framework to improve on the results in \cite{nagy2006kronecker}.  In that paper, the authors consider a banded/block-banded triply Toeplitz matrix with Dirichlet boundary conditions on all three levels,  whose `central' column is determined from a 3D discrete point-spread function (PSF).  Recognizing that the 3D PSF is a third-order tensor, they use third-order tensor decompositions to determine an approximate operator that is a (sum of) Kronecker products of Toeplitz matrices, and subsequently use the decomposition to
find a preconditioner for the deblurring problem. 

In their approach, the authors did not have a result that related the tensor approximation problem directly to matrix approximation problem.  We now show that using our framework, we can relate a (5th order) tensor approximation problem comprised of entries from a 3D {\it weighted} PSF, and tie these approximation problems together in a natural way, thus improving on the approach in \cite{nagy2006kronecker}.

To simplify discussion, we will assume that the blurring operator is applied to a $K \times K \times K$ image, where $K$ is odd.  Thus the matrix is an $K^3 \times K^3$ blurring operator. The matrix is assumed Toeplitz at each level, and
that the upper and lower bandwidth on all levels is $\frac{K-1}{2}$.  
Because of the 3-level banded Toeplitz structure, the central column is sufficient to specify every entry in the 
blurring operator.  We can therefore consider the central column in the form of a third order tensor PSF, $\TP$, such that the 
$(\frac{K+1}{2},\frac{K+1}{2},\frac{K+1}{2})$ entry corresponds to the entry of the central column on the main diagonal in $\MA$.   The authors
of \cite{nagy2006kronecker} propose to use
a truncated HOSVD $\TP \approx \sum_i \sum_j \sum_k  \gamma_{i,j,k} (\Vu_i \circ \V{v}_j \circ \V{w}_k)$ to form $\MA \approx \sum_{i,j,k} \M{C}_k
\otimes \M{D}_j \otimes \M{F}_i$, where each matrix is Toeplitz.  For example, 
the Toeplitz matrix $\M{C}_k$ is formed using vector $\gamma_{i,j,k}^{1/3} \V{w}_k$ as its central column, and so forth.

The tensor that is decomposed by the process above {\it is} the 3D PSF, with no weighting. This makes sense as a first choice since the entries in the PSF give the entries in the matrix, but it is not guaranteed that the matrix approximation error is equal to the tensor approximation error.  Under our framework, we arrive at the optimally weighted PSF tensor naturally. 
In our framework applied to this problem we have $p_1=p_2=p_3 = N$, $L=3$, $m = n = 1$ so that our tensor is of size $1 \times N \times N \times N \times 1$.  
Specifically, our approach gives\footnote{Note the swapping of the indices in the first and third modes.}
\[ \TX = [\cTT_{\mc{E}}[\MA]]_{1,i_1,i_2,i_3,1} = \sqrt{\eta_{i_1} \eta_{i_2} \eta_{i_3} }  \TP_{i_3,i_2,i_1} , \]
ensuring that our the matrix approximation error is equal to the tensor approximation error in approximating $\TX$, as given in \cref{lem:multiapprox}.  Since two modes are singleton dimensions, our 5th order tensor is effectively a third order tensor, and the matrices $\MX, \MZ$ are just $1 \times 1$ (i.e., $\MX = \MZ = 1$) in \cref{eq:multiTX}.  If $\TX \approx \T{T} = \T{G} \times_1 1 \times_2 \MU \times_3 \MV \times_4 \MW \times_5 1,$ then
                              using \cref{eq:defineH}, we have 
\[ \MH^{(i_1,i_2,i_3)} = \sum_{j_1}^{k_1} \sum_{j_2}^{k_2} \sum_{j_3}^{k_3} u_{i_1,j_1} v_{i_2,j_2} w_{i_3,j_3} \TG_{1,j_1,j_2,j_3,1}. \]

We therefore arrive at the matrix approximation
\[ \mc{M}_{\mc{E}}[\T{T}] = \sum_{j_1=1}^{k_1} \sum_{j_2=1}^{k_2}
\sum_{j_3=1 }^{k_3} \TG_{1,j_1,j_2,j_3,1}  \structLa(\MU_{:,j_1})
\otimes \structLb(\MV_{:,j_2}) \otimes \structLc(\MW_{:,j_3}),
\] 
 and the scalars $\TG_{1,j_1,j_2,j_3,1}$ can be distributed into the Kronecker factors as desired. 
 Thus, while the method in \cite{nagy2006kronecker} and our framework both give a sum of structured Kronecker products, our new approach has the optimal scaling baked into the process to ensure that the corresponding matrix approximation relates directly to the tensor approximation \footnote{The difference in the index correspondence between the two methods has to do with the orientation of the tensor, where we reverse the first and third modes from their definition of $\MP$.}.

\section{Conclusions and Future Work} \label{sec:final}
We have provided a powerful new tensor-based framework for
computing structure-preserving matrix approximations, and demonstrated how the framework can be applied to great effect on some problems in scientific computing.  Additionally, we showed that our framework is flexible enough that it can be used to describe optimal structured matrix approximation problems found in previous literature.  The key design elements of our approach include the invertible matrix-to-tensor mapping followed by a suitable choice of tensor decomposition.  For both the CP and Tucker formats, we showed how the features of the respective tensor approximations map back to the structure in the corresponding matrix approximation.  We proved that the matrix approximations derived using our matrix-to-tensor mapping are as good as the corresponding
tensor approximations, when the error is measured in the Frobenius norm.  Moreover, we proved that it is possible to preserve symmetric positive definiteness in the matrix approximation. 

Although we investigated the details given when using CP and Tucker decomposition, our framework permits the use of any suitable tensor decomposition -- for example a tensor-train approach \cite{oseledets2011tensor} or a non-negativity constrained CP -- in the intermediate stage.  We are currently investigating the extension of the present approach to the case of dense matrices. Another possible area for investigation is whether or not the framework could be modified to handle the case of non-uniform block sizes on a given level.

\bibliography{refs}

\begin{thebibliography}{10}

\bibitem{ahmadi2020randomized}
S.~Ahmadi-Asl, A.~Cichocki, A.~H. Phan, I.~Oseledets, S.~Abukhovich, and
  T.~Tanaka.
\newblock Randomized algorithms for computation of {T}ucker decomposition and
  higher order {SVD} ({HOSVD}).
\newblock {\em arXiv preprint arXiv:2001.07124}, 2020.

\bibitem{amestoy2017complexity}
P.~Amestoy, A.~Buttari, J.-Y. l'Excellent, and T.~Mary.
\newblock On the complexity of the block low-rank multifrontal factorization.
\newblock {\em SIAM Journal on Scientific Computing}, 39(4):A1710--A1740, 2017.

\bibitem{batselier2018computing}
K.~Batselier, W.~Yu, L.~Daniel, and N.~Wong.
\newblock Computing low-rank approximations of large-scale matrices with the
  tensor network randomized {SVD}.
\newblock {\em SIAM Journal on Matrix Analysis and Applications},
  39(3):1221--1244, 2018.

\bibitem{che2019randomized}
M.~Che and Y.~Wei.
\newblock Randomized algorithms for the approximations of {T}ucker and the
  {T}ensor {T}rain decompositions.
\newblock {\em Advances in Computational Mathematics}, 45(1):395--428, 2019.

\bibitem{chow2014preconditioned}
E.~Chow and Y.~Saad.
\newblock Preconditioned {K}rylov subspace methods for sampling multivariate
  {G}aussian distributions.
\newblock {\em SIAM Journal on Scientific Computing}, 36(2):A588--A608, 2014.

\bibitem{de2000multilinear}
L.~De~Lathauwer, B.~De~Moor, and J.~Vandewalle.
\newblock A multilinear singular value decomposition.
\newblock {\em SIAM journal on Matrix Analysis and Applications},
  21(4):1253--1278, 2000.

\bibitem{GarveyMengNagy2018}
C.~Garvey, C.~Meng, and J.~G. Nagy.
\newblock Singular value decomposition approximation via {K}ronecker summations
  for imaging applications.
\newblock {\em SIAM Journal on Matrix Analysis and Applications},
  39(4):1836--1857, 2018.

\bibitem{gneiting2006geostatistical}
T.~Gneiting, M.~G. Genton, and P.~Guttorp.
\newblock Geostatistical space-time models, stationarity, separability, and
  full symmetry.
\newblock {\em Monographs On Statistics and Applied Probability}, 107:151,
  2006.

\bibitem{hackbusch2012tensor}
W.~Hackbusch.
\newblock {\em Tensor spaces and numerical tensor calculus}, volume~42.
\newblock Springer, Heidelberg, 2012.

\bibitem{halko2011finding}
N.~Halko, P.~G. Martinsson, and J.~A. Tropp.
\newblock Finding structure with randomness: probabilistic algorithms for
  constructing approximate matrix decompositions.
\newblock {\em SIAM Review}, 53(2):217--288, 2011.

\bibitem{juang1985eigensystem}
J.-N. Juang and R.~S. Pappa.
\newblock An eigensystem realization algorithm for modal parameter
  identification and model reduction.
\newblock {\em Journal of guidance, control, and dynamics}, 8(5):620--627,
  1985.

\bibitem{KammNagy}
J.~Kamm and J.~G. Nagy.
\newblock Optimal {K}ronecker product approximation of block {T}oeplitz
  matrices.
\newblock {\em SIAM Journal on Matrix Analysis and its Applications},
  22:155--172, 2000.

\bibitem{kilmer2013third}
M.~E. Kilmer, K.~Braman, N.~Hao, and R.~C. Hoover.
\newblock Third-order tensors as operators on matrices: {A} theoretical and
  computational framework with applications in imaging.
\newblock {\em SIAM Journal on Matrix Analysis and Applications},
  34(1):148--172, 2013.

\bibitem{kolda2009tensor}
T.~G. Kolda and B.~W. Bader.
\newblock Tensor decompositions and applications.
\newblock {\em SIAM review}, 51(3):455--500, 2009.

\bibitem{kolodziej2019suitesparse}
S.~P. Kolodziej, M.~Aznaveh, M.~Bullock, J.~David, T.~A. Davis, M.~Henderson,
  Y.~Hu, and R.~Sandstrom.
\newblock The {SuiteSparse} matrix collection website interface.
\newblock {\em Journal of Open Source Software}, 4(35):1244, 2019.

\bibitem{kramer2016tangential}
B.~Kramer and S.~Gugercin.
\newblock Tangential interpolation-based eigensystem realization algorithm for
  {MIMO} systems.
\newblock {\em Mathematical and Computer Modelling of Dynamical Systems},
  22(4):282--306, 2016.

\bibitem{kung1978new}
S.-Y. Kung.
\newblock A new identification and model reduction algorithm via singular value
  decomposition.
\newblock In {\em Proc. 12th Asilomar Conf. on Circuits, Systems and Computers,
  Pacific Grove, CA, November, 1978}, 1978.

\bibitem{litvinenko2020hlibcov}
A.~Litvinenko, R.~Kriemann, M.~G. Genton, Y.~Sun, and D.~E. Keyes.
\newblock {HLIBCov}: Parallel hierarchical matrix approximation of large
  covariance matrices and likelihoods with applications in parameter
  identification.
\newblock {\em MethodsX}, 7:100600, 2020.

\bibitem{minster2020efficient}
R.~Minster, A.~K. Saibaba, J.~Kar, and A.~Chakrabortty.
\newblock Efficient randomized algorithms for subspace system identification.
\newblock {\em arXiv preprint arXiv:2003.11872}, 2020.

\bibitem{minster2020randomized}
R.~Minster, A.~K. Saibaba, and M.~E. Kilmer.
\newblock Randomized algorithms for low-rank tensor decompositions in the
  {T}ucker format.
\newblock {\em SIAM Journal on Mathematics of Data Science}, 2(1):189--215,
  2020.

\bibitem{nagy2006kronecker}
J.~G. Nagy and M.~E. Kilmer.
\newblock Kronecker product approximation for preconditioning in
  three-dimensional imaging applications.
\newblock {\em IEEE Transactions on Image Processing}, 15(3):604--613, 2006.

\bibitem{olshevsky2006tensor}
V.~Olshevsky, I.~Oseledets, and E.~Tyrtyshnikov.
\newblock Tensor properties of multilevel {T}oeplitz and related matrices.
\newblock {\em Linear algebra and its applications}, 412(1):1--21, 2006.

\bibitem{oseledets2011tensor}
I.~Oseledets, E.~Tyrtyshnikov, and N.~Zamarashkin.
\newblock Tensor-train ranks for matrices and their inverses.
\newblock {\em Computational Methods in Applied Mathematics}, 11(3):394--403,
  2011.

\bibitem{oseledets2010approximation}
I.~V. Oseledets.
\newblock Approximation of $2^d\times 2^d$ matrices using tensor decomposition.
\newblock {\em SIAM Journal on Matrix Analysis and Applications},
  31(4):2130--2145, 2010.

\bibitem{oseledets2012solution}
I.~V. Oseledets and S.~V. Dolgov.
\newblock Solution of linear systems and matrix inversion in the {TT}-format.
\newblock {\em SIAM Journal on Scientific Computing}, 34(5):A2718--A2739, 2012.

\bibitem{pestana2019preconditioners}
J.~Pestana.
\newblock Preconditioners for symmetrized {T}oeplitz and multilevel {T}oeplitz
  matrices.
\newblock {\em SIAM Journal on Matrix Analysis and Applications},
  40(3):870--887, 2019.

\bibitem{savas2011clustered}
B.~Savas and I.~S. Dhillon.
\newblock Clustered low rank approximation of graphs in information science
  applications.
\newblock In {\em Proceedings of the 2011 SIAM International Conference on Data
  Mining}, pages 164--175. SIAM, 2011.

\bibitem{savas2016clustered}
B.~Savas and I.~S. Dhillon.
\newblock Clustered matrix approximation.
\newblock {\em SIAM Journal on Matrix Analysis and Applications},
  37(4):1531--1555, 2016.

\bibitem{sun2020low}
Y.~Sun, Y.~Guo, C.~Luo, J.~Tropp, and M.~Udell.
\newblock Low-rank tucker approximation of a tensor from streaming data.
\newblock {\em SIAM Journal on Mathematics of Data Science}, 2(4):1123--1150,
  2020.

\bibitem{tucker1966some}
L.~R. Tucker.
\newblock Some mathematical notes on three-mode factor analysis.
\newblock {\em Psychometrika}, 31(3):279--311, 1966.

\bibitem{van2000ubiquitous}
C.~F. Van~Loan.
\newblock The ubiquitous {K}ronecker product.
\newblock {\em Journal of computational and applied mathematics},
  123(1-2):85--100, 2000.

\bibitem{verhaegen2007filtering}
M.~Verhaegen and V.~Verdult.
\newblock {\em Filtering and system identification: a least squares approach}.
\newblock Cambridge university press, 2007.

\bibitem{wang2019block}
R.~Wang, Y.~Li, M.~W. Mahoney, and E.~Darve.
\newblock Block basis factorization for scalable kernel evaluation.
\newblock {\em SIAM Journal on Matrix Analysis and Applications},
  40(4):1497--1526, 2019.

\bibitem{xing2018preserving}
X.~Xing and E.~Chow.
\newblock Preserving positive definiteness in hierarchically semiseparable
  matrix approximations.
\newblock {\em SIAM Journal on Matrix Analysis and Applications},
  39(2):829--855, 2018.

\end{thebibliography}
\bibliographystyle{abbrv}
\end{document}